\documentclass{amsart}

\usepackage{amsmath,amsfonts,amsthm,amssymb,amsxtra,enumerate,color,times,euscript, latexsym}
\usepackage[active]{srcltx}

      \def\dC{{\mathbb C}}

   \def\dN{{\mathbb N}}   
      \def\dR{{\mathbb R}}

   \def\dZ{{\mathbb Z}}

      \def\cF{{\mathcal F}}
\def\cG{{\mathcal G}}      
      
   \def\cN{{\mathcal N}}

\def\bR{{\mathbf R}}  

\def\supp{\operatorname{supp}}
\def\ran{\operatorname{ran}}
\def\dom{\operatorname{dom}}
\def\span{\operatorname{span}}
\def\Im{\operatorname{Im}}

\def\f{\varphi}

\def\downbar#1{
\setbox10=\hbox{$#1$}
   \dimen10=\ht10 \advance\dimen10 by 2.5pt
   \ifdim \dimen10<15pt 
      \advance\dimen10 by -0.5pt
      \dimen11=\dimen10
      \advance\dimen10 by 2.5pt
      \lower \dimen11
   \else \lower \ht10 \fi
   \hbox {\hskip 1.5pt \vrule height \dimen10 depth \dp10}\relax}
 \def\upbar#1{
 \setbox10=\hbox{$#1$}
    \dimen10=\ht10 \advance\dimen10 by \dp10 \advance\dimen10 by 2.5pt
    \ifdim \dimen10<15pt 
       \advance\dimen10 by 2pt \fi
    \raise 2.5pt \hbox {\hskip -1.5pt \vrule height \dimen10}\relax}
\def\cfr#1#2{
 \downbar{#2} \hskip -1.5pt {\; #1 \; \over \thinspace \  #2}\upbar{#1}}

\newtheorem{theorem}{Theorem}[section]
\newtheorem{proposition}[theorem]{Proposition}
\newtheorem{corollary}[theorem]{Corollary}
\newtheorem{lemma}[theorem]{Lemma}
\newtheorem{definition}[theorem]{Definition}
\theoremstyle{definition}

\newtheorem{remark}[theorem]{Remark}

\numberwithin{equation}{section}

\begin{document}
\title[The Jacobi matrices approach to Nevanlinna-Pick problems]
{The Jacobi matrices approach to Nevanlinna-Pick problems}
\author{Maxim  Derevyagin}

\address{Department of Mathematics MA 4-5\\
Technische Universit\"at Berlin\\
Strasse des 17. Juni 136\\
D-10623 Berlin\\
Germany}

\email{derevyagin.m@gmail.com}
\date{\today}

\subjclass{Primary 47A57, 47B36; Secondary 41A21, 42C05.}
\keywords{Jacobi matrix, linear pencil, Nevanlinna-Pick problem, 
multipoint Pad\'e approximants.}

\begin{abstract}
A modification of the well-known step-by-step process
for solving Nevan\-linna-Pick problems in the class of $\bR_0$-functions
gives rise to a linear pencil $H-\lambda J$, where 
 $H$ and $J$ are Hermitian tridiagonal matrices. 
First, we show that  $J$ is a positive operator.
Then it is proved that the corresponding Nevanlinna-Pick problem has a unique solution
iff the densely defined symmetric operator $J^{-\frac{1}{2}}HJ^{-\frac{1}{2}}$ is self-adjoint
and some criteria for this operator to be self-adjoint are presented. Finally,
by means of the operator technique, we obtain that multipoint diagonal Pad\'e approximants  to a unique solution $\varphi$
of the Nevanlinna-Pick problem converge to $\varphi$ locally uniformly in $\dC\setminus\dR$.
The proposed scheme extends the classical Jacobi matrix approach to moment problems
and Pad\'e approximation for $\bR_0$-functions.
\end{abstract}

\maketitle

\section{Introduction}
The connection with Jacobi matrices has led to numerous applications 
of spectral techniques for self-adjoint operators in the theory of moment problems, orthogonal polynomials
on the real line, and Pad\'e approximation.
Let us recall some basic ideas of this interplay.
First, note that one of the key tools in relating these theories is
the class $\bR_0$ of all functions having the representation
\begin{equation}\label{Mar_f}
\varphi(\lambda)=\int_{\dR}\frac{d\sigma(t)}{t-\lambda},
\end{equation}
where $\sigma$ is a probability measure, that is,
$\int_{\dR}{d\sigma(t)}=1$. 
If the support $\supp\sigma$ of $\sigma$ is contained in $[\alpha,\beta]$ 
we will say that $\varphi\in\bR[\alpha,\beta]$.

Consider a probability measure $\sigma$ such that all the moments
\begin{equation}\label{sMP}
s_n:=\int_{\dR}t^nd\sigma(t),\quad n\in\dZ_+:=\dN\cup\{0\}
\end{equation}
are finite. In this case, the corresponding function $\f$ has the following
asymptotic expansion
\begin{equation}\label{MP_np}
\f(\lambda)=-\frac{s_{0}}{\lambda}-\frac{s_{1}}{\lambda^2}
-\dots-\frac{s_{2n}}{\lambda^{2n+1}}+o\left(\frac{1}{\lambda^{2n+1}}\right),
\quad\lambda\widehat{\rightarrow }\infty,
\end{equation}
for every $n\in\dZ_+$ (here and throughout in the sequel $\lambda\widehat{\rightarrow }\infty$ means
that $\lambda$ tends to $\infty$ non-tangentially, that is, inside the
sector $\varepsilon<\arg \lambda<\pi-\varepsilon$ for some
$\varepsilon>0$). In view of the Hamburger-Nevanlinna theorem~\cite{A}, 
the classical moment problem reads as follows.

\noindent{\bf Hamburger moment problem.} 
Is the function $\f\in\bR_0$ satisfying~\eqref{MP_np} uniquely determined 
by the sequence $\{s_j\}_{j=0}^{\infty}$ of moments?

The moment problem is called determinate if $\f$ is uniquely determined. Otherwise
the moment problem is said to be indeterminate.
In fact, one can give an answer to the question in terms of the underlying Jacobi operators 
generated by Jacobi matrices. To see Jacobi matrices
in this context, note that one can expand $\varphi$ into the following
continued fraction
\begin{equation}\label{Jfraction}
\varphi(\lambda)=-\frac{1}{\lambda-a_0-\displaystyle{\frac{b_0^2}
{\lambda-a_1-\displaystyle{\frac{b_1^2}{\ddots}}}}}=
-\cfr{1}{\lambda-a_0}
-\cfr{b_0^2}{\lambda-a_1}
-\cfr{b_{1}^2}{\lambda-a_2}
-\cdots,
\end{equation}
where $a_j$ are real numbers, $b_j$ are positive numbers (see~\cite{A},~\cite{wall},~\cite{NikSor}).
Moreover, numbers $a_j$ and $b_j$ can be explicitly expressed in terms of the moments
$s_0,\dots,s_{2j+1}$~\cite{A}.
Continued fractions of the form~\eqref{Jfraction} are called J-fractions~\cite{JT},~\cite{wall}.
To the continued fraction~\eqref{Jfraction} one can associate a Jacobi matrix $H$ 
and its truncation $H_{[0,n-1]}$
\[
H=\left(%
\begin{array}{cccc}
  a_0 & b_0 &  &  \\
  b_0 & a_1 & b_1 &  \\
      & b_1 & a_2 & \ddots \\
      &     & \ddots & \ddots \\
\end{array}%
\right),\quad
H_{[0,n-1]}=\left(%
\begin{array}{cccc}
  a_0 & b_0    &  &  \\
  b_0 & a_1    & \ddots &  \\
      & \ddots & \ddots & b_{n-2} \\
      &        & b_{n-2}& a_{n-1}     \\
\end{array}%
\right).
\]
Let $\ell^2_{[0,\infty)}$ denote a Hilbert space of complex
square summable sequences $(x_0,x_1,\dots)$ equipped with the inner
product
\[
(x,y)=\sum_{i=0}^{\infty}x_i\overline{y}_i,\quad x,y\in\ell^2_{[0,\infty)}.
\]
Now, in the standard way, we can define a minimal closed operator 
$H$ acting in $\ell^2$ generated by the matrix $H$~\cite{A},~\cite{Be}.
We will denote the domain of $H$ and the range of $H$ by
$\dom H$ and $\ran H$, respectively. It is easy to see that $H$ is symmetric, i.e.
\[
(Hx,y)=(x,Hy),\quad x,y\in\dom H.
\]
Moreover, it is well known that $H$ is self-adjoint if and only if the corresponding moment problem
is determinate and the solution of the problem admits the representation
\[
\varphi(\lambda)=\left((H-\lambda)^{-1}e_0,e_0\right)
\]
where $e=(1,0,\dots)^\top$ is a column vector (see~\cite{A},~\cite{S98}). 
In the indeterminate case, a description of all $\f\in\bR_0$
satisfying~\eqref{MP_np} can be found in~\cite{A}, \cite{Berg}, \cite{S98} 
(see also~\cite{DM} where the operator approach to truncated moment problems was proposed).
In both cases, we have
\[
-\frac{Q_n(\lambda)}{P_n(\lambda)}=\left((H_{[0,n-1]}-\lambda)^{-1}e_0,e_0\right)=
-\cfr{1}{\lambda-a_0}-
\dots
-\cfr{b_{n-2}^2}{\lambda-a_{n-1}},
\]
where $P_n$ are orthogonal polynomials with respect 
to $\sigma$, and $Q_n$ are polynomials of the second kind (see~\cite{A},~\cite{NikSor},~\cite{S98}).
It is an elementary fact of the continued fraction theory 
(see, for instance,~\cite{A},~\cite{BGM},~\cite{JT}) that
\begin{equation}\label{dPA}
\varphi(\lambda)+\frac{Q_n(\lambda)}{P_n(\lambda)}=
O\left(\frac{1}{\lambda^{2n+1}}\right),\quad \lambda\widehat{\rightarrow }\infty.
\end{equation}
In other words, relation~\eqref{dPA} means that the rational function $-Q_n/P_n$ is
the nth diagonal Pad\'e approximant to $\varphi$ at $\infty$ 
(for more details on Pad\'e approximants see~\cite{BGM}). Now, we see that
in the self-adjoint case, convergence of diagonal Pad\'e approximants appears as 
the strong resolvent convergence of the finite matrix approximations $H_{[0,n]}$
to $H$.
So, if the moment problem is determinate then the corresponding diagonal Pad\'e approximants
converge to the solution $\f$ locally uniformly in $\dC\setminus\dR$. 
This statement for the class $\bR[\alpha,\beta]$ is known as the Markov theorem~\cite{NikSor}.
The above-described scheme has been recently extended to the case of rational 
perturbations of Nevanlinna functions~\cite{D09}, \cite{DD}, \cite{DD09}.
Also, the scheme was adapted to the case of complex Jacobi matrices~\cite{Beck} 
and generalized to the case of band matrices~\cite{Martinez99}.

The main goal of this paper is to generalize the scheme to the case of Nevanlinna-Pick problems 
and to prove convergence of related multipoint diagonal Pad\'e approximants. 
To show our purpose more precisely, let us recall that 
the classical Hamburger moment problem is the limiting case of the following problem 
(see~\cite{A},~\cite{Garnett},~\cite{KN}).

\noindent{\bf Nevanlinna-Pick problem.} Let $\{z_k\}_{k=0}^{\infty}$ be a sequence
of distinct numbers from the upper half plane $\dC_+$
and let $\f\in\bR_0$. Define numbers $w_j:=\f(z_j)$. Is the function $\f\in\bR_0$ satisfying
the interpolation relation $\f(z_j)=w_j$, $j\in\dZ_+$, uniquely determined by the given data 
$\{z_k\}_{k=0}^{\infty}$, $\{w_k\}_{k=0}^{\infty}$?

In view of the classical uniqueness theorem for analytic functions, the answer to this question is trivial
if the sequence $\{z_k\}_{k=0}^{\infty}$ has at least one accumulation point in $\dC_+$.
So, {\bf in what follows we will suppose that the sequence $\{z_k\}_{k=0}^{\infty}$ does not have any accumulation point
in $\dC_+$}. In other words, all the accumulation points of the sequence $\{z_k\}_{k=0}^{\infty}$ lie in $\dR$.

Similarly to the moment problem case, 
the Nevanlinna-Pick problem is called determinate, if $\f$ is uniquely determined. Otherwise
the Nevanlinna-Pick problem is said to be indeterminate.
We should also note that diagonal Pad\'e approximants at $\infty$ are the limiting case
of the following multipoint diagonal Pad\'e approximants.
\begin{definition}[\cite{BGM}]\label{PadeInt}
The  nth multipoint diagonal Pad\'e approximant  for the function $\varphi$ at
the points $\{z_0,\overline{z}_0,\dots,z_j,\overline{z}_j,\dots\}$ is defined as a ratio
$-Q_n/P_n$ 
of two polynomials $Q_n$, $P_n$ of degree at most $n-1$ and $n$, respectively, such that
the function $P_n\f+Q_n$ vanishes at the points 
$z_0$, $\overline{z}_0$, \dots, $z_{n-1}$, $\overline{z}_{n-1}$.
\end{definition}
It appears that the problem of finding 
multipoint diagonal Pad\'e approximants for the $\bR_0$-function $\f$ at the points
$\{z_0,\overline{z}_0,\dots,z_j,\overline{z}_j,\dots\}$
 is closely related to a 
continued fraction expansion of the following type
\begin{equation}\label{R2}
-\cfr{1}{a^{(2)}_0\lambda-a_0^{(1)}}-
    \cfr{b_0^2(\lambda-z_0)(\lambda-\overline{z}_0)}{a^{(2)}_1\lambda-a_1^{(1)}}-
    \cfr{b_{1}^2(\lambda-z_1)(\lambda-\overline{z}_1)}{a^{(2)}_2\lambda-a_2^{(1)}}-
    \dots,
\end{equation}
where $a^{(1)}_j$ are real numbers and $a_j^{(2)}$, $b_j$ are positive numbers.
This continued fraction gives rise to a tridiagonal linear pencil $H-\lambda J$,
where $H$ and $J$ are semi-infinite tridiagonal matrices~\cite{DZ09} (see also~\cite{Zhe_BRF} where
tridiagonal linear pencils associated with general continued fractions of type~\eqref{R2}
were introduced). 
In this paper,  we firstly obtain that $J$ generates a positive operator. Then we introduce a densely defined symmetric operator
$J^{-\frac{1}{2}}HJ^{-\frac{1}{2}}$ and present criteria for this operator to be self-adjoint.
Next, we prove that the Nevannlina-Pick problem in question has a unique solution if and only if $J^{-\frac{1}{2}}HJ^{-\frac{1}{2}}$ is self-adjoint.
Finally, we show that if $J^{-\frac{1}{2}}HJ^{-\frac{1}{2}}$ is self-adjoint then the locally uniform convergence of 
the multipoint diagonal Pad\'e approximants
\[
-\frac{Q_{n+1}(\lambda)}{P_{n+1}(\lambda)}=\left((J_{[0,n]}^{-\frac{1}{2}}H_{[0,n]}J_{[0,n]}^{-\frac{1}{2}}-\lambda)^{-1}
J_{[0,n]}^{-\frac{1}{2}}e_0,J_{[0,n]}^{-\frac{1}{2}}e_0\right)
\]
to the unique solution 
\[
\varphi(\lambda)=\left((J^{-\frac{1}{2}}HJ^{-\frac{1}{2}}-\lambda )^{-1}J^{-\frac{1}{2}}e_0,J^{-\frac{1}{2}}e_0\right)
\]
of the Nevanlinna-Pick problem arises as the resolvent convergence.

The paper is organized as follows. In Section~2 we present the step-by-step process for solving 
the Nevanlinna-Pick problems and associated sequences of polynomials. 
In Section~3, a tridiagonal linear pencil is introduced and 
basic properties of the operator $J$ are given. 
The one-to-one correspondence between tridiagonal linear pencils and the Nevanlinna-Pick problems in question
is shown in Section~4. The next session is concerned with the Weyl circles.
Section~6 reveals the underlying symmetric operators. In Section~7, we characterize
the determinacy of the underlying Nevanlinna-Pick problems in terms
of the self-adjointness of $J^{-\frac{1}{2}}HJ^{-\frac{1}{2}}$.
After that, in Section~8, for the determinate case, we prove the locally uniform convergence
of multipoint diagonal Pad\'e approximants for $\bR_0$-functions.

\section{The modified multipoint Schur algorithm}
As is known, the Schur transformation is a powerful tool in solving moment and interpolation
problems (see~\cite{A}, \cite{ADL}).
The starting point of our analysis is the following modification of the Schur transformation.

\begin{proposition}[cf. \cite{DZ09}]\label{schur_step}
Let $\varphi\in\bR_0$ and let $z\in\dC_+$ be a
fixed number. Then there exist unique numbers $a^{(1)},\,a^{(2)}\in\dR$
and $b>0$ such that the function $\f_1$ defined by the equality
\begin{equation}\label{flt_step}
\f(\lambda)=-\frac{1}{a^{(2)}\lambda-a^{(1)}+b^2(\lambda-z)(\lambda-\overline{z})\f_1(\lambda)}
\end{equation}
belongs to $\bR_0\cup\{0\}$, that is, $\varphi_1$ has the representation~\eqref{Mar_f} with a probability measure
in case $\varphi_1\not\equiv 0$. Moreover, we have that
\begin{equation}\label{f_b}
b^2=a^{(2)}-1.
\end{equation}
\end{proposition}
\begin{proof}
To see that the numbers $a^{(1)}$, $a^{(2)}$ are
uniquely determined,
let us substitute $\lambda$ for $z$ and $\overline{z}$
in~\eqref{flt_step}. We thus get
\begin{equation}\label{f_w}
a^{(2)}z-a^{(1)}=-\frac{1}{\f(z)},\quad
a^{(2)}\overline{z}-a^{(1)}=-\frac{1}{\f(\overline{z})}.
\end{equation}
Eliminating from the above relations $a^{(1)}$ and $a^{(2)}$, one
can obtain the following formulas
\begin{equation}\label{f_a}
a^{(1)}=\left(\int_{\dR}\frac{td\sigma(t)}{|t-z|^2}\right)
\left|\int_{\dR}\frac{d\sigma(t)}{t-z}\right|^{-2},
\quad
a^{(2)}=\left(\int_{\dR}\frac{d\sigma(t)}{|t-z|^2}\right)
\left|\int_{\dR}\frac{d\sigma(t)}{t-z}\right|^{-2}.
\end{equation}
Further, it follows from the Schwartz lemma that 
\begin{equation}\label{flt_t}
\widetilde{\f_1}(\lambda)=-\frac{\frac{1}{\f(\lambda)}+a^{(2)}\lambda-a^{(1)}}{(\lambda-z)(\lambda-\overline{z})}=
\int_{\dR}\frac{d\mu(t)}{t-\lambda}
\end{equation}
(the proof of this fact is in line with 
that of \cite[Lemma~3.1]{DZ09}).
Choosing $b>0$ in the following way
\[
b^2=\int_{\dR}d\mu(t)
\]
and defining $\f_1:=\widetilde{\f}_1/b^2$ we get that 
the function $\f_1$ possesses the integral
representation~\eqref{Mar_f} with a probability measure.
 Finally, by taking $\lambda=iy$ and $y\to\infty$ 
in~\eqref{flt_t} we get~\eqref{f_b}.
\end{proof}

\begin{remark}
It should be noted that for $\f\in\bR[\alpha,\beta]$ this modification of the Schur algorithm 
was presented in~\cite[Lemma~3.1]{DZ09}. However, its proof is valid for $\f\in\bR_0$. 
A similar transformation for Caratheodory functions was proposed in~\cite{DG}.
\end{remark}

Let $\varphi$ be a non-rational function of the class $\bR_0$, i.e.  $\varphi$ admits 
the representation~\eqref{Mar_f} with a probability measure which has an infinite support.
Let also an infinite sequence $\{z_k\}_{k=0}^{\infty}\subset\dC_+$ of distinct numbers be given. 
Since $\f$ is not rational the given data give rise to infinitely
many steps of the step-by-step process. So, we have infinitely
many linear fractional transformations of the
form~\eqref{flt_step} which lead to the following continued
fraction
\begin{equation}\label{ContF}
-\cfr{1}{a^{(2)}_0\lambda-a_0^{(1)}}-
    \cfr{b_0^2(\lambda-z_0)(\lambda-\overline{z}_0)}{a^{(2)}_1\lambda-a_1^{(1)}}-
    \cfr{b_{1}^2(\lambda-z_1)(\lambda-\overline{z}_1)}{a^{(2)}_2\lambda-a_2^{(1)}}-
    \dots
\end{equation}
(for more details, see~\cite{DZ09}). It should be noted that general continued fractions associated with finding 
multipoint Pad\'e approximants were introduced in~\cite{HN89mp} and studied 
in~\cite{HN89},~\cite{IM95}.

It is immediate from the construction that the $(n+1)$th convergent 
of~\eqref{ContF} 
\[
-\frac{Q_{n+1}(\lambda)}{P_{n+1}(\lambda)}=
-\cfr{1}{a^{(2)}_0\lambda-a_0^{(1)}}
-\cdots -
\cfr{b_{n-1}^2(\lambda-z_{n-1})(\lambda-\overline{z}_{n-1})}{a^{(2)}_n\lambda-a_n^{(1)}}
\]
satisfies the following interpolation relation
\begin{equation}\label{IP}
\f(z_j)=-\frac{Q_{n+1}(z_j)}{P_{n+1}(z_j)}, \quad j=0,\dots,n.
\end{equation} 
Since $\f\in\bR_0$ and the coefficients $a^{(1)}_j$, $a^{(2)}_j$, $b_j$ are real, one also has
\[
\f(\overline{z}_j)=-\frac{Q_{n+1}(\overline{z}_j)}{P_{n+1}(\overline{z}_j)},\quad j=0,\dots,n.
\]
So, we have just concluded the following.
\begin{proposition}\label{MP_con} 
The rational function $-Q_{n+1}/P_{n+1}$ is the (n+1)th multipoint diagonal
Pad\'e approximant to $\varphi$ at the points $\{z_0,\overline{z}_0,\dots,z_j,\overline{z}_j,\dots\}$.
\end{proposition}
It is well known that denominators and numerators of convergents of a continued fraction satisfy
a three-term recurrence relation (see, for instance,~\cite{JT}). In particular, for
the continued fraction~\eqref{ContF} the recurrence relation takes the
following form
\begin{equation}\label{rec_rel}
u_{j+1}-(a_j^{(2)}\lambda-a_j^{(1)})u_{j}+b_{j-1}^2(\lambda-z_{j-1})(\lambda-\overline{z}_{j-1})u_{j-1}=0,
\quad j\in\dN.
\end{equation}
Further, the polynomials $P_j$ of the first kind  are solutions
$u_j=P_j(\lambda)$ of the system~\eqref{rec_rel} with the initial
conditions
\begin{equation}\label{InConP}
u_{0}=1,\quad u_1=a_0^{(2)}\lambda-a_0^{(1)}.
\end{equation}
Similarly, the polynomials of the second kind $Q_j(\lambda)$ are solutions $u_j=Q_j(\lambda)$ of the
system~\eqref{rec_rel} subject to the following initial conditions
\begin{equation}\label{InConP2}
u_{0}=0,\quad u_1=-1.
\end{equation}
\begin{remark}
Note that the polynomials $P_j$ are 
orthogonal with respect to the varying measures 
$\displaystyle{ \frac{d\sigma(t)} {\prod_{k=0}^{j-1}|t-z_k|^2} }$ 
(see~\cite{Gon78},~\cite{Lo78},~\cite[Section~6.1]{StTo}).
Moreover, for $\f\in\bR[\alpha,\beta]$ an operator treatment of the relation 
of the polynomials $P_j$ to orthogonal rational functions 
was presented in~\cite{DZ09} (see~\cite[Section~9.5]{Bul}, where this relation is also discussed).
It should be also remarked that some orthogonality relations for polynomials and rational 
functions related to general continued fractions of type~\eqref{ContF} 
were obtained in~\cite{IM95},~\cite{Zhe_BRF}
(see also~\cite{Zh_PadeTable}, where biorthogonality properties of rational functions related to
multipoint Pad\'e approximation were studied and concrete examples connected with generalized
hypergeometric functions were constructed).
\end{remark}

\section{Tridiagonal linear pencils associated with $\bR_0$-functions}

In order to see linear pencils in our context, let us note that 
the recurrence relation~\eqref{rec_rel} can be renormalized to the
following one
\begin{equation}\label{r_rl}
(\overline{\mathfrak{b}}_{j-1}-\lambda{\mathfrak{d}}_{j-1})\widehat{u}_{j-1}+
(\mathfrak{a}_j-\lambda\mathfrak{c}_j)\widehat{u}_{j}+
(\mathfrak{b}_j-\lambda\mathfrak{d}_j)\widehat{u}_{j+1}=0,\quad j\in\dN,
\end{equation}
where the numbers $\mathfrak{a}_j$, $\mathfrak{b}_j$, $\mathfrak{c}_j$, $\mathfrak{d}_j$
are defined as follows  
\[
\mathfrak{a}_j=a_j^{(1)},\quad \mathfrak{b}_j={z}_jb_j,\quad
\mathfrak{c}_j=a_j^{(2)},\quad \mathfrak{d}_j=b_j,\quad j\in\dZ_+,
\]
and the transformation $u\to\widehat{u}$ has the following form
\begin{equation}\label{wTransform}
\widehat{u}_0=u_0,\quad
\widehat{u}_j={\displaystyle\frac{u_j}{b_0\dots
b_{j-1}(z_0-\lambda)\dots(z_{j-1}-\lambda)}},\quad j\in\dN.
\end{equation}
Thus, we have two associated sequences $\widehat{P}_j$ and $\widehat{Q}_j$ of rational functions obtained from the polynomial sequences 
${P}_j$ and ${Q}_j$, respectively, by means of the transformation~\eqref{wTransform}.
In contrast to the polynomial case, the rational functions $\widehat{P}_j$ are not orthogonal with respect to the original
measure $\sigma$ since 
\begin{equation}\label{helpfororthR1}
\int_{\dR}\widehat{P}_0(t)\widehat{P}_1(t)d\sigma(t)=
\int_{\dR}\widehat{P}_1(t)d\sigma(t)=1-a_0^{(2)}
\end{equation}
and, due to~\eqref{f_a}, $1-a_0^{(2)}\ne 0$ for any $z_0\in\dC_+$.
Despite this, some orthogonality properties remain valid (see~\cite[Theorem~2.10]{BDZh}).
It should be also noted that some orthogonal proper rational functions satisfy a relation similar 
to~\eqref{r_rl}~\cite[p.~541]{Atk} (see also~\cite{Bul} for the recurrence relations for orthogonal rational functions). 

The relation~\eqref{r_rl} naturally leads to a linear pencil $H-\lambda J$, where
\[
H=\left(%
\begin{array}{cccc}
  \mathfrak{a}_0 & \mathfrak{b}_0 &  &  \\
 \overline{\mathfrak{b}}_0 & \mathfrak{a}_1 & \mathfrak{b}_1 &  \\
      & \overline{\mathfrak{b}}_1 & \mathfrak{a}_2 & \ddots \\
      &     & \ddots & \ddots \\
\end{array}%
\right),\quad\
J=\left(%
\begin{array}{cccc}
  \mathfrak{c}_0 & \mathfrak{d}_0 &  &  \\
  \mathfrak{d}_0 & \mathfrak{c}_1 & \mathfrak{d}_1 &  \\
      & \mathfrak{d}_1 & \mathfrak{c}_2 & \ddots \\
      &     & \ddots & \ddots \\
\end{array}%
\right)
\]
are Jacobi matrices. 
For an infinite matrix $A$, we denote by $A_{[j,k]}$ the square sub-matrix
obtained by taking rows and columns $l=j,j+1,...,k \le\infty$. For example,
for finite $j$ and $k$ we have that
\[
H_{[j,k]}=
\begin{pmatrix}
  \mathfrak{a}_j & \mathfrak{b}_j & {\bf 0} \\
  \overline{\mathfrak{b}}_j & \ddots &    \\
      {\bf 0}&  & \mathfrak{a}_k\\
      \end{pmatrix},\quad
      J_{[j,k]}=
\begin{pmatrix}
  \mathfrak{c}_j & \mathfrak{d}_j & {\bf 0} \\
  \mathfrak{d}_j & \ddots &    \\
      {\bf 0}&  & \mathfrak{c}_k\\
      \end{pmatrix}.
\]
By $J$ we also denote the minimal closed operator on $\ell^2_{[0, \infty)}$ generated by the matrix $J$~\cite{A}.
Obviously, $J$ is a symmetric operator. Besides, due to~\eqref{f_b}, we have the relation
$\mathfrak{c}_j=1+\mathfrak{d}_j^2$, which gives us the following factorization of $J$
\begin{equation}\label{mainfac}
J=L^*L=\begin{pmatrix}
  1 & \mathfrak{d}_0 &  &  \\
  0 & 1 & \mathfrak{d}_1 &  \\
      & 0 & 1 & \ddots \\
      &     & \ddots & \ddots \\
\end{pmatrix}
\begin{pmatrix}
  1 & 0 &  &  \\
  \mathfrak{d}_0 & 1 &0 &  \\
      & \mathfrak{d}_1 & 1 & \ddots \\
      &     & \ddots & \ddots \\
\end{pmatrix}.
\end{equation}
The factorization of $J$ allows us to say a bit more about $J$. 
\begin{proposition}\label{positive}
 The operator $J$ is self-adjoint and positive, that is,
\[
(Jx,x)>0,\quad x\in\dom J\setminus\{0\}.
\]
In particular, $\ker J=\{0\}$.
\end{proposition}
\begin{proof}
Let us consider the Hermitian form $\left(J\xi,\xi\right)$ on finitely supported sequences $\xi$, that is,
$\xi=({\xi}_0,{\xi}_1,\dots,{\xi}_n,0,0,\dots)^{\top}$.
By virtue of~\eqref{mainfac}, we have that
\[
(J\xi,\xi)=(L\xi,L\xi)\ge 0.
\]
Further, let us prove that  $\ker J^{*}=\{0\}$.
Suppose the converse, that is, there exists $\eta\in\ell^2$ such that
$J^*\eta=0$ and $\eta\ne0$. Taking into account the structure of $J$ we get the equality
\[
0=(J^*\eta,\eta)=|{\eta}_0|^2+|\mathfrak{d}_0\eta_0+{\eta}_1|^2+\dots+
|\mathfrak{d}_{n-1}\eta_{n-1}+{\eta}_{n}|^2+\dots,
\]
which implies $\eta=0$. So, $\ker J=\ker J^*=\{0\}$. This contradiction also shows that
\begin{equation}\label{nonsf}
\sum_{k=0}^{\infty}|p_k(0)|^2=\infty,
\end{equation}
where $p_j$ are polynomials of the first kind associated with $J$.
Since the relation~\eqref{nonsf} doesn't hold true for Jacobi operators with deficiency indices (1,1) (see~\cite{Be},~\cite{S98}),
we obtain that $J$ is self-adjoint.
The statement of the proposition also immediately follows from~\cite[Theorem~VII.1.4]{Be}.
\end{proof}

\begin{remark}
It has been recently proved~\cite{BDZh} that if $\f\in\bR[\alpha,\beta]$ and $z_k\to\infty$ then
\[
(Jx,x)\ge\delta(x,x),\quad x\in\ell^2,
\]
for some $\delta>0$.
Furthermore, in this case the operator $J$ is a compact perturbation of $I$ and, in fact,
the linear pencil $H-\lambda J$ is a compact perturbation
of the classical pencil $H_0-\lambda I$ 
(which corresponds to the limiting case $z_k=\infty$ for $k=0,1,2,\dots$). 
It should be noted that in the case of orthogonal Laurent polynomials
a similar tridiagonal pencil was considered in~\cite{CKvA}. 
Roughly speaking, the case of orthogonal Laurent polynomials corresponds to the multiple
interpolation at 0 and $\infty$, which is known as the strong moment problem on the real line~\cite{JT}.
An operator approach to the strong moment problem was given in~\cite{Hen87}. It is also worth to note that,
in the matrix case, Jacobi type symmetric operators related to the matrix strong moment problems
were presented and studied in~\cite{Simonov1},~\cite{Simonov2}. 
\end{remark}

Since $\ker J=\{0\}$ and $J$ is self-adjoint, we can consider the self-adjoint operator $J^{-\frac{1}{2}}$,
which is not necessarily bounded. However, the following statement holds true.
 
\begin{proposition}
We have that 
\begin{equation}\label{ForRoot1}
e_j\in\dom J^{-\frac{1}{2}},\quad j\in\dZ_+,
\end{equation}
where the vectors $e_0=(1,0,0,\dots)^{\top}$, $e_1=(0,1,0,\dots)^{\top}$, $\dots$  form the standard basis in~$\ell^2$.
\end{proposition}
\begin{proof}
It is the basic spectral theory that for the positive operator $J$ 
there exists a resolution of the identity $E_t$ such that 
\[
Jf=\int_{0}^{\infty}tdE_tf,\quad f\in\dom J,
\]
and $f\in\dom J$ if and only if $\int_{0}^{\infty}t^2d(E_tf,f)<\infty$~\cite[Section 66]{AG}. Moreover, we also have that
\[
J^{-\frac{1}{2}}f=\int_{0}^{\infty}\frac{1}{\sqrt{t}}dE_tf,\quad f\in\dom J^{-\frac{1}{2}},
\]
and $f\in\dom J^{-\frac{1}{2}}$ if and only if $\int_{0}^{\infty}\frac{1}{t}d(E_tf,f)<\infty$. Now,
\eqref{ForRoot1} is equivalent to
\[
\int_{0}^{\infty}\frac{1}{t}d(E_te_j,e_j)<\infty,\quad j\in\dZ_+.
\]
First we will prove that
\begin{equation}\label{mainfordomain}
\int_{0}^{\infty}\frac{1}{t}d(E_te_0,e_0)<\infty.
\end{equation}
For simplicity, let us denote $\nu=(E_{\cdot}e_0,e_0)$ and introduce the similar measures $\nu_n=(E_{\cdot}^{(n)}e_0,e_0)$ for the truncations 
$J_{[0,n]}$, where $E_{\cdot}^{(n)}$ is such that
\[
J_{[0,n]}=\int_{0}^{\infty}tdE_{t}^{(n)}, \quad n\in\dZ_+.
\]
Next, it is a standard fact of theory of moment problems~\cite{A}, that 
\[
\int_{0}^{\infty}\psi(t)d\nu_n(t)\to\int_{0}^{\infty}\psi(t)d\nu(t), \quad n\to\infty,
\]
for any simple function  $\psi$ (that is, $\psi$ is measurable and assumes only a finite number of values).
Now, recall that in~\cite[Lemma~6.1]{DZ09} it was proved  that 
\begin{equation}\label{dom_inq_tr}
\int_{0}^{\infty}\frac{1}{t}d\nu_n(t)=\left(J_{[0,n]}^{-1}e_0,e_0\right)\le 1,\quad n\in\dZ_+.
\end{equation}
Thus, Fatou's lemma for varying measures~\cite[Proposition~17, p.~231]{Royden} and~\eqref{dom_inq_tr} yield
\begin{equation}\label{h_p61}
\int_{0}^{\infty}\frac{1}{t} d\nu(t)\le\liminf_{n\to\infty}\int_{0}^{\infty}\frac{1}{t}d\nu_n(t)\le 1.
\end{equation}
The rest is a consequence of~\eqref{mainfordomain}. Indeed, it is well known that for any $\lambda$ from the resolvent set $\rho(J)$
of the operator $J$ we have the following formula for the diagonal Green function
\begin{equation}\label{Greenformula}
\left((J-\lambda)^{-1}e_j,e_j\right)=p_j(\lambda)\left(p_j(\lambda)\left((J-\lambda)^{-1}e_0,e_0\right)+q_j(\lambda)\right),\quad
j\in\dZ_+,
\end{equation}
where $p_j$ and $q_j$ are polynomials of the first and second kinds, respectively, associated with the Jacobi matrix $J$
 (see for example~\cite[Theorem~2.10]{Beck},~\cite[Proposition~2.2]{GS}).
Putting $\lambda=-x$, $x>0$, into formula~\eqref{Greenformula}, it can be rewritten as follows
\[
\int_{0}^{\infty}\frac{1}{t+x}d(E_te_j,e_j)=p_j(-x)\left(p_j(-x)\int_{0}^{\infty}\frac{1}{t+x}d\nu(t)+q_j(-x)\right),\quad j\in\dN,
\]
where $p_j(-x)=\frac{\det (J_{[0,j-1]}+x)}{\mathfrak{d}_0\dots\mathfrak{d}_{j-1}}>0$ for $x\ge 0$. Now, it remains to apply 
the Fatou lemma to $\int_{0}^{\infty}\frac{1}{t+x}d(E_te_j,e_j)$ as $x\to 0$ and to use~\eqref{mainfordomain}.
\end{proof}

\begin{remark}\label{dom_J_h}
The main ingredient in the proof was to obtain~\eqref{mainfordomain}. Another way to prove it is through 
the Darboux transformations. Namely, let us consider a Jacobi matrix $J_1=LL^*$ and let $\nu^*$ be 
a corresponding probability measure associated with $J_1$. Then it follows from~\cite[Theorem~3.4]{BM04}
that
\[
d\nu(t)=ctd\nu^*(t),\quad c>0.
\]
The latter relation immediately implies~\eqref{mainfordomain}.
\end{remark}
To end this section, note that we can now say more about the sequence $\left(J_{[0,n]}^{-1}e_0,e_0\right)$.
Namely, the following relation holds true
\begin{equation}\label{h_con_d}
\left(J_{[0,n]}^{-1}e_0,e_0\right)\to 1,\quad\text{as}\quad
n\to\infty.
\end{equation}
Indeed, by applying~\cite[formula~(2.15)]{GS} we see that
$\left(J_{[0,n]}^{-1}e_0,e_0\right)$, $n\in\dZ_+$, are convergents to
the continued fraction
\begin{equation*}
\cfr{1}{\mathfrak{c}_0}-
    \cfr{\mathfrak{d}_0}{\mathfrak{c}_1}-
    \cfr{\mathfrak{d}_1}{\mathfrak{c}_2}-
    \dots.
\end{equation*}
Forasmuch as $\mathfrak{c}_j=1+\mathfrak{d}_j^2$, applying 
the remark to \'Sleszy\'nski-Pringsheim's theorem given on \cite[p.~93]{JT} implies~\eqref{h_con_d}.

\section{Relations between Nevanlinna-Pick problems and linear pencils}
In this section we show that there exists a one-to-one correspondence between the linear pencils
under consideration and the Nevanlinna-Pick problems in question. We also re-examine some facts
for the polynomials $P_{j}$ and $Q_{j}$ which are well known for orthogonal polynomials.

We begin with  the following connection between the polynomials of the
first and second kinds $P_{j}$, $Q_{j}$ and the truncated
linear pencils $\lambda J_{[0,j]}-H_{[0,j]}$, which in the classical case can be found
in~\cite[Section 7.1.2]{Be} and~\cite[Section 6.1]{Atk}.
\begin{proposition}\label{detf}
The polynomials $P_{j}$ and $Q_j$, $j\in\dN$,
can be found by the formulas
\begin{equation}\label{det_f_r2}
P_{j}(\lambda)=\det(\lambda J_{[0,j-1]}-H_{[0,j-1]}),\quad
Q_{j}(\lambda)=\det(\lambda J_{[1,j-1]}-H_{[1,j-1]}).
\end{equation}
The zeros of the polynomials $P_{j}$ and $Q_j$ are real. 
Moreover, the polynomials $P_{j}$ and $Q_j$ do not have common zeros.
\end{proposition}
\begin{proof}
Formula~\eqref{det_f_r2} immediately follows 
from the definition of $P_{j}$ and $Q_{j}$ by using the Laplace expansions 
of the determinants by the last row. Since $J_{[0,j-1]}$ is
strictly positive, one can rewrite the first relation in~\eqref{det_f_r2} as follows
\[
P_{j}(\lambda)=\det J_{[0,j-1]}^{1/2}\det(\lambda - J_{[0,j-1]}^{-1/2}H_{[0,j-1]}J_{[0,j-1]}^{-1/2})
\det J_{[0,j-1]}^{1/2}.
\]
Clearly, $J_{[0,j-1]}^{-1/2}H_{[0,j-1]}J_{[0,j-1]}^{-1/2}$ is a self-adjoint matrix. Thus, 
the latter relation yields the fact that the zeros of $P_j$ are real. Similarly, one can
show that the zeros of $Q_j$ are real. The last statement follows by induction via applying
the Laplace expansion of the determinant $\det(\lambda J_{[0,j-1]}-H_{[0,j-1]})$ by the first row.
\end{proof}

By induction, one easily gets from~\eqref{r_rl} the Liouville-Ostrogradsky formula
\begin{equation}\label{Ostrogr2}
   Q_{n+1}(\lambda)P_{n}(\lambda) -
        Q_{n}(\lambda) P_{n+1}(\lambda)= 
\prod_{k=0}^{n-1}b_k^2(\lambda-z_k)(\lambda-\overline{z}_k), \quad     
\end{equation}
for every $n\in\dZ_+$ (see~\cite{BDZh}). Going further in this direction, we should note
that, sometimes, it is very useful to have~\eqref{r_rl} in the following matrix form
\begin{equation}\label{forCD1}
(H-\lambda J)\pi_{[0,j]}(\lambda)=-(\mathfrak{b}_j-\lambda\mathfrak{d}_j)\widehat{P}_{j+1}(\lambda)e_j+
(\overline{\mathfrak{b}}_j-\lambda{\mathfrak{d}}_j)\widehat{P}_{j}(\lambda)e_{j+1},
\end{equation}
\begin{equation}\label{forCD2}
(H-\lambda J)\xi_{[0,j]}(\lambda)=-(\mathfrak{b}_j-\lambda\mathfrak{d}_j)\widehat{Q}_{j+1}(\lambda)e_j+
(\overline{\mathfrak{b}}_j-\lambda{\mathfrak{d}}_j)\widehat{Q}_{j}(\lambda)e_{j+1}+e_0,
\end{equation}
where the vectors $\pi_{[0,j]}(\lambda)$ and $\xi_{[0,j]}(\lambda)$ are defined as follows
\[
\pi_{[0,j]}(\lambda)=\left(\widehat{P}_0(\lambda),
\widehat{P}_1(\lambda),\dots,\widehat{P}_j(\lambda),0,0,\dots\right)^{\top},
\]
 \[
\xi_{[0,j]}(\lambda)=\left(\widehat{Q}_0(\lambda),
\widehat{Q}_1(\lambda),\dots,\widehat{Q}_j(\lambda),0,0,\dots\right)^{\top}. 
 \]
For example, by virtue of~\eqref{forCD1} we get the following generalization of the Christoffel-Darboux formula.
\begin{proposition} We have that for $j\in\dZ_+$
\begin{equation}\label{CDmain}
\begin{split}
(\lambda-\overline{\zeta})\sum_{k=0}^{j}(\widehat{P}_k(\lambda)+\mathfrak{d}_{k-1}\widehat{P}_{k-1}(\lambda))
\overline{(\widehat{P}_k(\zeta)+\mathfrak{d}_{k-1}\widehat{P}_{k-1}(\zeta))}=\\
=\frac{P_{j+1}(\lambda)\overline{P_j({\zeta})}-\overline{P_{j+1}({\zeta})}P_j(\lambda)}
{\prod_{k=0}^{j-1}b_k^2(\lambda-z_k)(\overline{\zeta}-\overline{z}_k)},
\end{split}
\end{equation}
where $\mathfrak{d}_{-1}=0$ for convenience and $\lambda,\zeta\in\dC_+\setminus\{z_k\}_{k=0}^{j}$.
\end{proposition}
\begin{proof}
It clearly follows from~\eqref{forCD1} that
\begin{equation}\label{helpCD1}
\left((H-\lambda J)\pi_{[0,j]}(\lambda), \pi_{[0,j]}(\zeta)\right)=-(\mathfrak{b}_j-\lambda\mathfrak{d}_j)
\widehat{P}_{j+1}(\lambda)\overline{\widehat{P}_j(\zeta)},
\end{equation}
\begin{equation}\label{helpCD2}
\left((H-\overline{\zeta}J)\pi_{[0,j]}(\lambda), \pi_{[0,j]}(\zeta)\right)=-(\overline{\mathfrak{b}}_j-\overline{\zeta}\mathfrak{d}_j)
\overline{\widehat{P}_{j+1}(\zeta)}{\widehat{P}_j(\lambda)}.
\end{equation}
Subtracting~\eqref{helpCD1} from~\eqref{helpCD2} and using~\eqref{wTransform} we get the following relation  
\begin{equation}\label{CDvar}
(\lambda-\overline{\zeta})\left(J\pi_{[0,j]}(\lambda),\pi_{[0,j]}({\zeta})\right)=
\frac{P_{j+1}(\lambda)P_j(\overline{\zeta})-P_{j+1}(\overline{\zeta})P_j(\lambda)}
{\prod_{k=0}^{j-1}b_k^2(\lambda-z_k)(\overline{\zeta}-\overline{z}_k)}.
\end{equation}
Now, observe that  due to~\eqref{mainfac} we have
\[
\left(J\pi_{[0,j]}(\lambda),\pi_{[0,j]}({\zeta})\right)=
\left(L\pi_{[0,j]}(\lambda),L\pi_{[0,j]}({\zeta})\right)
\]
and, so, from~\eqref{CDvar} we obtain~\eqref{CDmain}.
\end{proof}

\begin{remark}To see how it is related to the classical Christoffel-Darboux relation~\cite{A} let us note that,
according to~\eqref{f_a} and~\eqref{f_b}, we have that
$\mathfrak{d}_k\to 0$ and $b_k^2/|z_k|^2\to\widetilde{b}_k^2\ne 0$ as $z_k\to\infty$, $k=0,\dots,j$
 provided that the numbers $\int_{\dR}t^kd\sigma(t)$ are finite for $k=0,\dots,j$.
Consequently, the classical Christoffel-Darboux formula is the limiting case of~\eqref{CDmain}. Moreover,
it is shown in~\cite[Theorem~2.2]{DZ09} (see also~\cite[Section~4]{BDZh}) that the sequence $\{\widehat{P}_k+\mathfrak{d}_{k-1}\widehat{P}_{k-1}\}_{k=0}^{\infty}$ is a sequence of 
rational functions orthogonal with respect to the original measure~$\sigma$ 
(see~\cite{Bul} for further information on orthogonal rational functions).
\end{remark}
In what follows we will also need the following relation
\begin{equation}\label{forKW}
 \begin{split}
\sum_{k=0}^{j}|\omega(\widehat{P}_k(\lambda)+\mathfrak{d}_{k-1}\widehat{P}_{k-1}(\lambda))+
\widehat{Q}_k(\lambda)+\mathfrak{d}_{k-1}\widehat{Q}_{k-1}(\lambda)|^2-
\frac{\omega-\overline{\omega}}{\lambda-\overline{\lambda}}=\\
=\left(J(\omega\pi_{[0,j]}(\lambda)+\xi_{[0,j]}(\lambda)),(\omega\pi_{[0,j]}(\lambda)+\xi_{[0,j]}(\lambda) )\right)-
\frac{\omega-\overline{\omega}}{\lambda-\overline{\lambda}}=\\
=\frac{1}{\Im\lambda}\frac{|\omega P_j(\lambda)+Q_j(\lambda)|^2}
{\prod_{k=0}^{j-1}b_k^2|\lambda-z_k|^2}\Im\frac{\omega P_{j+1}(\lambda)+Q_{j+1}(\lambda)}{\omega P_j(\lambda)+Q_j(\lambda)},
 \end{split}
\end{equation}
where $\omega\in\dC_+$ and $\lambda\in\dC_+\setminus\{z_k\}_{k=0}^{j}$.  Formula~\eqref{forKW} can be easily obtained
by straightforward manipulations with~\eqref{forCD1} and~\eqref{forCD2} (for the classical case see~\cite[Section~I.2.1]{A}).

Next, by following~\cite{GS}, let us introduce $m$-functions of the truncated linear pencils.
\begin{definition}
Let $j$ and $n$ be nonnegative integers such that $j\le n$. The function
\begin{equation}\label{Weyl1}
m_{[j,n]}(\lambda)=\left((H_{[j,n]}-\lambda
J_{[j,n]})^{-1}e_j,e_j\right)
\end{equation}
will be called the $m$-function of the linear pencil
$H_{[j,n]}-\lambda J_{[j,n]}$.
\end{definition}

To see the correctness of the above given definition it is sufficient to
recall that $J_{[j,n]}$ is positive definite in view of~Proposition~\ref{positive} and to rewrite~\eqref{Weyl1} in the following form
\begin{equation}\label{cor_Weyl}
m_{[j,n]}(\lambda)=\left((J_{[j,n]}^{-\frac{1}{2}}H_{[j,n]}J_{[j,n]}^{-\frac{1}{2}}-\lambda)^{-1}
J_{[j,n]}^{-\frac{1}{2}}e_j,J_{[j,n]}^{-\frac{1}{2}}e_j\right).
\end{equation}
 Literally as in the classical case (see for instance~\cite{GS}),
one obtains that $m$-functions satisfy the Riccati equation.

\begin{proposition}[\cite{DZ09}]\label{pr_ric}
The $m$-functions $m_{[j,n]}$ and $m_{[j+1,n]}$ are related by the
equality
\begin{equation}\label{Riccati}
m_{[j,n]}=-\frac{1}{a^{(2)}_j\lambda-a^{(1)}_j+b^2_j(\lambda-z_j)(\lambda-\overline{z}_j)m_{[j+1,n]}(\lambda)}.
\end{equation}
\end{proposition}

The latter statement allows us to see the relation of $m$-functions to multipoint diagonal Pad\'e approximants.
\begin{proposition}\label{interlace}
Let $\theta_n=\det J_{[0,n]}/\det J_{[1,n]}$ and $\eta_n=\det J_{[0,n]}/\det J_{[0,n-1]}$. Then the
function $\theta_nm_{[0,n]}$ is an $\bR_0$-function and 
\begin{equation}\label{mformulas}
m_{[0,n]}(\lambda)=-\frac{Q_{n+1}(\lambda)}{P_{n+1}(\lambda)},
\end{equation}
that is, $m_{[0,n]}$ is the $(n+1)$th multipoint 
diagonal Pad\'e approximant for $\f$.
Moreover, we have that $-\eta_nP_{n}/P_{n+1}\in\bR_0$.
\end{proposition}
\begin{proof}
Formula~\eqref{mformulas} is implied by the relation~\eqref{Riccati}. 
Now, from Proposition~\ref{MP_con} we see that $m_{[0,n]}$ is the $(n+1)$th multipoint 
diagonal Pad\'e approximant for $\f$.  To see that $\theta_nm_{[0,n]}\in\bR_0$, it is enough to recall that
$\Phi\in\bR_0$ if and only if 
\[
\frac{\Im\Phi(\lambda)}{\Im\lambda}>0,\quad \lambda\in\dC\setminus\dR,
\]
and $\sup\limits_{y>0}|y\Phi(iy)|=1$~\cite[Section~III.1.1]{A}. The first condition is easily verified by means of~\eqref{cor_Weyl}
and the second one follows from~\eqref{mformulas}.
In the same way,  by noticing that
\[
-\frac{P_n(\lambda)}{P_{n+1}(\lambda)}=-
\frac{\det(\lambda J_{[0,n-1]}-H_{[0,n-1]})}{\det(\lambda J_{[0,n]}-H_{[0,n]})}=
\left((H_{[0,n]}-\lambda
J_{[0,n]})^{-1}e_n,e_n\right)
\]
one can check that $-\eta_nP_{n}/P_{n+1}\in\bR_0$ since $\eta_n>0$.
\end{proof}
Due to $-\theta_nQ_{n+1}/P_{n+1}\in\bR_0$ and $-\eta_nP_{n}/P_{n+1}\in\bR_0$, we get the following.  
\begin{corollary}
We have that
\begin{enumerate}
\item[i)] The zeros of $Q_{n+1}$ and $P_{n+1}$  interlace,
\item[ii)] The zeros of $P_{n}$ and $P_{n+1}$ interlace.
\end{enumerate}
\end{corollary}

Summing up Propositions~\ref{schur_step} and~\ref{interlace},
 we conclude the following.
\begin{theorem}
There is a one-to-one correspondence between the linear pencils in question and the data
$\{z_k\}_{k=0}^{\infty}$, $\{w_k\}_{k=0}^{\infty}$ of the Nevanlinna-Pick problems.
\end{theorem}
\begin{proof}
It follows from formulas~\eqref{f_b} and~\eqref{f_a} that the data $\{z_k\}_{k=0}^{\infty}$, $\{w_k\}_{k=0}^{\infty}$
uniquely determine the linear pencil, that is, the following numbers
\begin{equation}\label{penreprN}
\mathfrak{a}_j=a_j^{(1)},\quad \mathfrak{b}_j={z}_jb_j,\quad
\mathfrak{c}_j=a_j^{(2)},\quad \mathfrak{d}_j=b_j,\quad j\in\dZ_+,
\end{equation}
where $a_j^{(1)}\in\dR$, $a_j^{(2)}>0$, $b_j>0$, $z_j\in\dC_+$, and $\mathfrak{c}_j=1+\mathfrak{d}_j^2$.
Let us suppose that we are given a set of numbers that can be represented as above. 
Then we see from~\eqref{penreprN}
that $z_j=\mathfrak{b}_j/\mathfrak{d}_j$. Finally, by virtue of Proposition~\ref{interlace} we get that
the numbers $w_j$ are uniquely determined by the formula
\[
w_j=-\frac{Q_n(z_j)}{P_n(z_j)}
\]
for large enough $n$. It remains to note that in view of the precompactness of the family $-{Q_n}/{P_n}$
(see Proposition~\ref{prec_R1}) and~\eqref{h_con_d} there exists a function $\varphi\in\bR_0$ which satisfies
the underlying interpolation relation $\varphi(z_j)=w_j$, $j\in\dZ_+$.
\end{proof}

\section{The Weyl circles}

The classical Weyl circles approach to Nevanlinna-Pick problems can be found in~\cite[Section~IV.6]{Garnett}. 
In this section, following~\cite[Section~I.2.3]{A}, we adapt the notion of the Weyl circles to the linear pencil case.

Let us  begin by considering the function
\begin{equation}\label{Ach7}
 \omega_j(\lambda,\tau)=-\frac{Q_j(\lambda)-\tau Q_{j-1}(\lambda)}{P_j(\lambda)-\tau P_{j-1}(\lambda)},
\end{equation}
where $\lambda\in\dC\setminus\dR$, $\tau\in\dR\cup\{\infty\}$, and $j\in\dN$. Obviously, 
from the definition we have that
\[ 
\omega_j(\lambda,\infty)=\omega_{j-1}(\lambda,0).
\] 
Moreover, in view of~\eqref{IP} we have that $\omega_j(z_k,\tau)=w_k$ and 
$\omega_j(\overline{z}_k,\tau)=\overline{w}_k$ for $j=k+2, k+3,\dots$.
So, formula~\eqref{Ach7} gives a parametrization of [j-1/j] rational solutions to the truncated
Nevanlinna-Pick problems. Another such a parametrization is given in~\cite[Theorem 6.1.3]{Bul} in terms of orthogonal
rational functions of the first and second kinds.

Due to Proposition~\ref{interlace},  the number $-\frac{P_{j-1}(\lambda)}{P_j(\lambda)}$ is not real for any 
$\lambda\in\dC\setminus\dR$ and, therefore, we see that the set
\[
K_{j}(\lambda)=\{\omega_j(\lambda,\tau): \tau\in\dR\cup\{\infty\}\}.
\]
is a circle. In addition, we have that $K_j(\overline{\lambda})=\overline{K_j(\lambda)}$. So,
we can consider only the case when $\lambda\in\dC_+$.
The following statement contains a characterization of the circle $K_j(\lambda)$.
\begin{theorem}
Let $\lambda\in\dC_+\setminus\{z_k\}_{k=0}^{j-1}$ be a fixed number. Then the center of $K_j(\lambda)$ is
\begin{equation}\label{Ach8}
-\frac{Q_j(\lambda)\overline{P_{j-1}(\lambda)}-Q_{j-1}(\lambda)\overline{P_{j}(\lambda)}}
{P_j(\lambda)\overline{P_{j-1}(\lambda)}-P_{j-1}(\lambda)\overline{P_{j}(\lambda)}},
\end{equation}
and the radius of $K_j(\lambda)$ is
\begin{equation}\label{Ach9}
\frac{1}{|\lambda-\overline{\lambda}|}\frac{1}{\sum_{k=0}^{j-1}|\widehat{P}_k(\lambda)+\mathfrak{d}_{k-1}\widehat{P}_{k-1}(\lambda)|^2}.
\end{equation}
Besides, the equation of $K_j(\lambda)$ can be represented as follows (setting $\mathfrak{d}_{-1}=0$)
\begin{equation}\label{Ach10}
\sum_{k=0}^{j-1}|\omega(\widehat{P}_k(\lambda)+\mathfrak{d}_{k-1}\widehat{P}_{k-1}(\lambda))+
\widehat{Q}_k(\lambda)+\mathfrak{d}_{k-1}\widehat{Q}_{k-1}(\lambda)|^2-
\frac{\omega-\overline{\omega}}{\lambda-\overline{\lambda}}=0.
\end{equation}
\end{theorem}
\begin{proof}
By the same reasoning as in the proof of~\cite[Theorem~1.2.3]{A} we conclude that
\[
\omega_j(\lambda,\tau)=-\frac{Q_j(\lambda)\overline{P_{j-1}(\lambda)}-Q_{j-1}(\lambda)\overline{P_{j}(\lambda)}}
{P_j(\lambda)\overline{P_{j-1}(\lambda)}-P_{j-1}(\lambda)\overline{P_{j}(\lambda)}}+
\left|\frac{Q_j(\lambda){P_{j-1}(\lambda)}-Q_{j-1}(\lambda){P_{j}(\lambda)}}
{P_j(\lambda)\overline{P_{j-1}(\lambda)}-P_{j-1}(\lambda)\overline{P_{j}(\lambda)}}\right| e^{i\theta},
\]
where $\theta=\theta(\tau)$ is real. The latter relation immediately gives us~\eqref{Ach8} and  the formula for the radius
of $K_j(\lambda)$
\[
\left|\frac{Q_j(\lambda){P_{j-1}(\lambda)}-Q_{j-1}(\lambda){P_{j}(\lambda)}}
{P_j(\lambda)\overline{P_{j-1}(\lambda)}-P_{j-1}(\lambda)\overline{P_{j}(\lambda)}}\right|,
\]
which by means of~\eqref{Ostrogr2} and~\eqref{CDmain} can be reduced to~\eqref{Ach9}.

The rest of the proof is identical to the proof of~\cite[Theorem~1.2.3]{A}.
\end{proof}

Denote by ${\bf K}_j(\lambda)$ the closure of the interior of $K_j(\lambda)$. Then the following statement holds true.
\begin{corollary}
Let $\lambda\in\dC_+\setminus\{z_k\}_{k=0}^{j-1}$ be a fixed number.  Then
the set ${\bf K}_j(\lambda)$ is a set of numbers $\omega\in\dC$ satisfying the inequality
\begin{equation}\label{WeylBall}
\sum_{k=0}^{j-1}|\omega(\widehat{P}_k(\lambda)+\mathfrak{d}_{k-1}\widehat{P}_{k-1}(\lambda))+
\widehat{Q}_k(\lambda)+\mathfrak{d}_{k-1}\widehat{Q}_{k-1}(\lambda)|^2\le
\frac{\omega-\overline{\omega}}{\lambda-\overline{\lambda}}.
\end{equation}
\end{corollary}
Furthermore, we can get a relation between the discs ${\bf K}_{j+1}(\lambda)$ and ${\bf K}_{j}(\lambda)$.
\begin{corollary}\label{Kembed}
We have that
\[
{\bf K}_{j+1}(\lambda)\subseteq{\bf K}_{j}(\lambda),\quad j\in\dN.
\]
Besides, the circles $K_{j+1}(\lambda)$ and $K_{j}(\lambda)$ have at least one common point.
\end{corollary}
\begin{proof}
The proof of the both corollaries is in line with the proof of the analogous statements given in~\cite[Section~2.3]{A}.
\end{proof}

Now, we see that there are two options for the sequence ${\bf K}_{j}(\lambda)$. Namely, we can have a limit point or a limit circle.
\begin{theorem}\label{limitpoint}
Let $\lambda\in\dC_+\setminus\{z_k\}_{k=0}^{\infty}$  be a fixed number. Then the sequence ${\bf K}_{j}(\lambda)$ converges to a point iff
\[
{\sum_{k=0}^{\infty}|\widehat{P}_k(\lambda)+\mathfrak{d}_{k-1}\widehat{P}_{k-1}(\lambda)|^2}=\infty.
\]
\end{theorem}
\begin{proof}
The proof is immediate from Corollary~\ref{Kembed} and~\eqref{Ach9}.
\end{proof}

Next, we obtain the existence of the Weyl solution.

\begin{theorem}\label{WeylSolution}
For every $\lambda\in\dC_+\setminus\{z_k\}_{k=0}^{\infty}$ there exists a number $\omega=\omega(\lambda)\in\dC_+$ such that
\begin{equation}\label{WeylSolIneq}
\sum_{k=0}^{\infty}|\omega(\widehat{P}_k(\lambda)+\mathfrak{d}_{k-1}\widehat{P}_{k-1}(\lambda))+
\widehat{Q}_k(\lambda)+\mathfrak{d}_{k-1}\widehat{Q}_{k-1}(\lambda)|^2\le
\frac{\omega-\overline{\omega}}{\lambda-\overline{\lambda}}.
\end{equation}
 \end{theorem}
\begin{proof}
The statement is a straightforward consequence of Corollary~\ref{Kembed} and the inequality~\eqref{WeylBall}.
\end{proof}

Finally, it should be noticed that the mentioned parametrization from~\cite{Bul} leads to a slightly different 
but very similar theory of nested disks~\cite[Section 10]{Bul}. That theory is equivalent to the presented one in the sense that
the underlying Nevanlinna-Pick Problems are the same. 

\section{The underlying symmetric operators}

In this section we reduce the linear pencil in question to an
operator generated by the formal matrix expression~$J^{-\frac{1}{2}}HJ^{-\frac{1}{2}}$. Namely, we show that 
this operator is a densely defined symmetric operator. 

Since $e_j\in\dom J\subset\dom J^{\frac{1}{2}}$ the vectors $f_j:=J^{\frac{1}{2}}e_j$, $j\in\dZ_+$, belong to $\ell^2$.
The relation $\ker J^{\frac{1}{2}}=\{0\}$ implies that the linear span 
\[
\cF=\span\{f_j\}_{j=0}^{\infty}=\left\{\sum_{k=0}^{n}c_kf_k: c_k\in\dC, n\in\dZ_+\right\}
\]
is dense in $\ell^2$.  
In view of~\eqref{ForRoot1}, we can also introduce the vectors $g_j:=J^{-\frac{1}{2}}e_j$, $j\in\dZ_+$, which
lie in $\ell^2$. Moreover, the linear span $\cG=\span\{g_j\}_{j=0}^{\infty}$ is dense in $\ell^2$. Besides, we have that
that the systems $\{f_j\}_{j=0}^{\infty}$ and  $\{g_j\}_{j=0}^{\infty}$ are bi-orthogonal, i.e. 
\[
(f_j,g_k)=\begin{cases}
0,&j\ne k,\\
1,& j=k.
\end{cases}
\]
As a consequence, we get that there is a one-to-one correspondence between $h\in\ell^2$ and the formal series
\[
\sum_{k=0}^{\infty}(h,g_k)f_k,\quad \sum_{k=0}^{\infty}(h,f_k)g_k.
\]
In this case, we will write $h\sim \sum_{k=0}^{\infty}(h,g_k)f_k$ or  $h\sim \sum_{k=0}^{\infty}(h,f_k)g_k$. 
Next, we see that  (setting $\mathfrak{b}_{-1}=0$ for convenience)
\[
J^{-\frac{1}{2}}HJ^{-\frac{1}{2}}f_j=\mathfrak{b}_{j-1}g_{j-1}+
\mathfrak{a}_jg_j+\overline{\mathfrak{b}}_jg_{j+1},\quad j\in\dZ_+.
\]
So, we have that $J^{-\frac{1}{2}}HJ^{-\frac{1}{2}}:\cF\mapsto\cG$.
Thus the domain of the matrix expression $J^{-\frac{1}{2}}HJ^{-\frac{1}{2}}$ is dense in $\ell^2$. 
\begin{proposition}\label{formalsymm}
The formal matrix expression $J^{-\frac{1}{2}}HJ^{-\frac{1}{2}}$ generates a densely defined symmetric operator with the deficiency
indices either (1,1) or (0,0).
\end{proposition}
 \begin{proof}
It is easy to see that
\[
(J^{-\frac{1}{2}}HJ^{-\frac{1}{2}}f_j,f_k)=(f_j,J^{-\frac{1}{2}}HJ^{-\frac{1}{2}}f_k),\quad j,k\in\dZ_+,
\]
that is, $J^{-\frac{1}{2}}HJ^{-\frac{1}{2}}$ is symmetric in $\ell^2$. Thus, the operator is closable and, in what follows,
by  $J^{-\frac{1}{2}}HJ^{-\frac{1}{2}}$ we denote the minimal closed operator defined by the matrix expression~$J^{-\frac{1}{2}}HJ^{-\frac{1}{2}}$.
Let $(J^{-\frac{1}{2}}HJ^{-\frac{1}{2}})^*$ be adjoint to $J^{-\frac{1}{2}}HJ^{-\frac{1}{2}}$ in $\ell^2$. 
By the definition, a vector $h\in\dom(J^{-\frac{1}{2}}HJ^{-\frac{1}{2}})^*$ 
if and only if there exists a vector $h^*\in\ell^2$ such that
\[
(J^{-\frac{1}{2}}HJ^{-\frac{1}{2}}f_k,h)=(f_k,h^*),\quad f\in k\in\dZ_+.
\]
Further, it can be rewritten as follows
\[
(\mathfrak{b}_{k-1}g_{k-1}+\mathfrak{a}_kg_k+\overline{\mathfrak{b}}_kg_{k+1}, h)=
(f_k,h^*),\quad k\in\dZ_+,
\]
which actually implies that
\[
y_k=\overline{\mathfrak{b}}_{k-1}x_{k-1}+\mathfrak{a}_kx_k+{\mathfrak{b}}_kx_{k+1},\quad k\in\dZ_+,
\]
where $h\sim\sum_{k=0}^{\infty}x_kf_k$ and $h^*\sim\sum_{k=0}^{\infty}y_kg_k$.
Thus, $h\in\dom(J^{-\frac{1}{2}}HJ^{-\frac{1}{2}})^*$ if and only if 
there exists $h^*\in\ell^2$ such that 
\[
h^*\sim\sum_{k=0}^{\infty}(\overline{\mathfrak{b}}_{k-1}x_{k-1}+\mathfrak{a}_kx_k+{\mathfrak{b}}_kx_{k+1})g_k.
\]
The next step is to determine the deficiency indices. In order to do that we should find nontrivial solutions of the equation 
\begin{equation}\label{indd}
((J^{-\frac{1}{2}}HJ^{-\frac{1}{2}})^*-\overline{\lambda} )h=0,\quad \Im\lambda\ne 0.
\end{equation}
Let $h\sim\sum_{k=0}^{\infty}x_kf_k$ be a solution to~\eqref{indd}. Then we obviously have that 
\[
(f_k, ((J^{-\frac{1}{2}}HJ^{-\frac{1}{2}})^*-\overline{\lambda} )h)=0,\quad k\in\dZ_+,
\]
which reduces to the following 
\[
\overline{\mathfrak{b}}_{k-1}\overline{x}_{k-1}+\mathfrak{a}_k\overline{x}_k+
{\mathfrak{b}}_k\overline{x}_{k+1}=
\lambda(f_k,h),\quad k\in\dZ_+.
\]
Observing that $(f_k,h)={\mathfrak{d}}_{k-1}\overline{x}_{k-1}+
\mathfrak{c}_k\overline{x}_k+{\mathfrak{d}}_k\overline{x}_{k+1}$, we arrive at
\[
(\overline{\mathfrak{b}}_{k-1}-\lambda {\mathfrak{d}}_{k-1})\overline{x}_{k-1}+
(\mathfrak{a}_k-\lambda{\mathfrak{c}}_{k})\overline{x}_k+
({\mathfrak{b}}_k-\lambda{\mathfrak{d}}_{k+1})\overline{x}_{k+1}=0,\quad k\in\dZ_+.
\]
In view of~\eqref{r_rl}, \eqref{wTransform}, and ~\eqref{InConP}, we conclude that
$\overline{x}_k=c\widehat{P}_k(\lambda)$. So, the linear space $\cN_{\lambda}$
of the solutions to~\eqref{indd} has dimension 1 if there exists an element $h\in\ell^2$ such that 
\begin{equation}\label{l2series}
h\sim\sum_{k=0}^{\infty}\overline{\widehat{P}_k(\lambda)}f_k.
\end{equation}
Otherwise,  the linear space $\cN_{\lambda}$ has dimension 0.

Let us find the condition for $h$ from~\eqref{l2series} to belong to $\ell^2$. 
First, we should check the weak convergence of the sequence $h_n=\sum_{k=0}^{n}\overline{\widehat{P}_k(\lambda)}f_k$.
Obviously, we have that $(h_n,g_k)\to(h,g_k)=\overline{\widehat{P}_k(\lambda)}$ as $n\to\infty$. Furthermore,
$\overline{\cG}=\overline{\span\{g_j\}_{j=0}^{\infty}}=\ell^2$. 
Consequently, according to the criterion  of the weak convergence we get that the convergence of~\eqref{l2series} 
is implied by the uniform boundedness of the following sequence
\begin{equation}\label{chvost}
\begin{split}
\Vert\sum_{k=0}^{n}\overline{\widehat{P}_k(\lambda)}f_k\Vert=(J\pi_{[0,n]}(\lambda),\pi_{[0,n]}(\lambda))=\\
=(L\pi_{[0,n]}(\lambda),L\pi_{[0,n]}(\lambda))=\sum_{k=0}^{n}|\widehat{P}_k(\lambda)+\mathfrak{d}_{k-1}\widehat{P}_{k-1}(\lambda)|^2.
\end{split}
\end{equation}
From~\eqref{chvost} we see that  the condition
\begin{equation}\label{l2shproof}
{\sum_{k=0}^{\infty}|\widehat{P}_k(\lambda)+\mathfrak{d}_{k-1}\widehat{P}_{k-1}(\lambda)|^2}<\infty
\end{equation}
guarantees the existence of $h$ satisfying~\eqref{l2series}. It turns out that this condition is also necessary. Indeed,
let us suppose the converse that ${\sum_{k=0}^{\infty}|\widehat{P}_k(\lambda)+\mathfrak{d}_{k-1}\widehat{P}_{k-1}(\lambda)|^2}=\infty$
and there exists $h\in\ell^2$ having the representation~\eqref{l2series}. Then it follows from~\eqref{indd} that $h\in\ran J^{-\frac{1}{2}}$ and, therefore,
$h=J^{\frac{1}{2}}h_0$ for some $h_0\in\ell^2$. The latter means that
\[
\Vert h\Vert=\Vert J^{\frac{1}{2}}h_0\Vert=\Vert Lh_0\Vert={\sum_{k=0}^{\infty}|\widehat{P}_k(\lambda)+\mathfrak{d}_{k-1}\widehat{P}_{k-1}(\lambda)|^2}=\infty,
\]
which yields the contradiction. So, $\dim\cN_{\lambda}=1$ if and only if~\eqref{l2shproof} holds true.

It is well known that for symmetric operators the deficiency index $d_{\lambda}=\dim\cN_{\lambda}$ is the same for
each $\lambda\in\dC_+$ as well as for each $\lambda\in\dC_-$. Further, it follows from~\eqref{Ach9} that 
\[
{\sum_{k=0}^{n-1}|\widehat{P}_k(\lambda)+\mathfrak{d}_{k-1}\widehat{P}_{k-1}(\lambda)|^2}=
{\sum_{k=0}^{n-1}|\widehat{P}_k(\overline{\lambda})+\mathfrak{d}_{k-1}\widehat{P}_{k-1}(\overline{\lambda})|^2}
\]
since the radii of $K_n(\lambda)$ and $K_n(\overline{\lambda})$  are equal.
The latter relation implies that $d_{\lambda}=d_{\overline{\lambda}}$.
\end{proof}

Now we are in a position to formulate criteria for $J^{-\frac{1}{2}}HJ^{-\frac{1}{2}}$ to be self-adjoint 
(for the classical case see~\cite{A},~\cite{Be},~\cite{S98}).
\begin{theorem}
The following statements are equivalent:
\begin{enumerate}
 \item[i)] The operator $J^{-\frac{1}{2}}HJ^{-\frac{1}{2}}$ is self-adjoint;
\item[ii)] The sequence ${\bf K}_j(\lambda)$ converges to a point for some $\lambda\in\dC_+\setminus\{z_k\}_{k=0}^{\infty}$;
\item[iii)] We have that 
\begin{equation}\label{limpoinser}
{\sum_{k=0}^{\infty}|\widehat{P}_k(\lambda)+\mathfrak{d}_{k-1}\widehat{P}_{k-1}(\lambda)|^2}=\infty
\end{equation}
for some $\lambda\in\dC_+\setminus\{z_k\}_{k=0}^{\infty}$.
\end{enumerate}
\end{theorem}
\begin{proof}
The equivalence of ii) and iii) is established in Theorem~\ref{limitpoint}.
The equivalence of i) and iii)  is actually proved in the proof of Proposition~\ref{formalsymm} by showing that the defect vector~\eqref{l2series}
belongs to $\ell^2$ if and only if~\eqref{limpoinser} holds true. 
\end{proof}

\begin{remark}
It is well known that for symmetric operators the dimension of the defect space $\cN_{\lambda}$ remains
the same for all $\lambda\in\dC_+$ . Thus, if~\eqref{limpoinser} holds for some $\lambda_0\in\dC_+\setminus\{z_k\}_{k=0}^{\infty}$
then it holds for all $\lambda\in\dC_+\setminus\{z_k\}_{k=0}^{\infty}$. The same is true for the limit point case.
\end{remark}

We should emphasize that in our approach the operator $J^{-\frac{1}{2}}HJ^{-\frac{1}{2}}$
plays exactly the same role as the Jacobi matrix for a moment problem. We should also stress here 
that if the original measure has finite moments of all nonnegative orders and 
we have a collection of interpolation sequences $\{z_k^{(n)}\}_{k=0}^{\infty}$ such that for every $k\in\dZ_+$
\[
z_{k}^{(n)}\to\infty,\quad\text{as}\quad n\to\infty,
\]
then the corresponding matrices $J^{(n)}$ converge to the identity $I$, as $n\to\infty$, elementwise 
(see~\eqref{f_b} and~\eqref{f_a}). 
So, roughly speaking, in this case, the operator $(J^{(n)})^{-\frac{1}{2}}H^{(n)}(J^{(n)})^{-\frac{1}{2}}$ approaches
the classical Jacobi matrix (see also~\cite{BDZh}).

To complete this section, it should be remarked that, in recent years, a lot of attention has been paid to 
the study of orthogonal polynomials on the unit circle via the spectral theory of
CMV-matrices (see~\cite{simon} and references therein).  
Roughly speaking, orthogonal polynomials on the unit circle correspond to 
the multiple interpolation problem at $0$ and $\infty$ for the Schur class
(actually, there is only one interpolation point since $\infty$ is symmetric to $0$
with respect to the unit circle). The multiple interpolation at two points is, in some sense, the limiting case
of the case under consideration. 
Also note  that an operator approach to orthogonal rational functions  on the unit circle 
via CMV matrices can be found in~\cite{vel}.
It is also worth mentioning that Jacobi type normal matrices associated to
complex moment problems were introduced and studied in~\cite{BD05}, \cite{BD06}.

\section{The uniqueness of Nevanlinna-Pick problems}

In this section, by mimicking 
the proofs of~\cite[Theorem~2.10]{S98} and~\cite[Theorem~2.11]{S98}, 
we characterize the determinacy of the Nevanlinna-Pick problems in question 
in terms of the self-adjointness of $J^{-\frac{1}{2}}HJ^{-\frac{1}{2}}$.

Let $\f\in\bR_0$ and let a sequence of distinct numbers 
$\{z_k\}_{k=0}^{\infty}\subset\dC_+$ be given.
According to~\eqref{f_w} and~\eqref{f_b}, the pencil $H-\lambda J$ in question is uniquely
determined by the sequences $\{z_k\}_{k=0}^{\infty}$ and $w_k:=\f(z_k)$, $k\in\dZ_+$.
So, as we already mentioned, the following question naturally arises.

\noindent{\bf Nevanlinna-Pick problem.} 
Is the function $\f\in\bR_0$ satisfying the interpolation relation
\begin{equation}\label{NP_ex}
\f(z_k)=w_k,\quad k\in\dZ_+
\end{equation}
uniquely determined by the data $\{z_k\}_{k=0}^{\infty}$, $\{w_k\}_{k=0}^{\infty}$?

More details about Nevanlinna-Pick problems can be found in~\cite{A},~\cite{Garnett},~\cite{KN}.
\begin{remark}
Recall that an $\bR$-function is a function which is holomorphic in the open upper
 half plane $\dC_+$ and maps $\dC_+$ onto $\dC_+$.
For convenience, it is supposed that every $\f\in\bR$ is extended 
to the lower half plane $\dC_-$ by the symmetry relation $\f(\lambda)=\overline{\f(\overline{\lambda})}$,
$\lambda\in\dC_-$. Clearly, $\bR_0$ is a subclass of $\bR$.
In fact, the condition $\f\in\bR_0$ means that $\f$ is an $\bR$-function and satisfies 
the following tangential interpolation condition
\begin{equation}\label{Ad_IP}
\f(\lambda)=-\frac{1}{\lambda}+
o\left(\frac{1}{\lambda}\right),
\quad\lambda\widehat{\rightarrow }\infty.
\end{equation}
Roughly speaking, \eqref{Ad_IP} can be interpreted as the interpolation conditions
$\f(\infty)=0$, $\f'(\infty)=-1$. So, the Nevanlinna-Pick problem in question is a sublass
of Nevanlinna-Pick problems in $\bR$.
\end{remark}

Before answering the question of the Nevanlinna-Pick problem we will prove the following
auxiliary statement. 

\begin{lemma}\label{lemma1}
We have that for $j\in\dZ_+$
\begin{equation}\label{h1}
\begin{split}
e_0&=(H-\overline{z}_jJ)
(\xi_{[0,j]}(\overline{z}_j)+m_{[0,j]}(\overline{z}_j)\pi_{[0,j]}(\overline{z}_j))=\\
&=(H_{[0,j]}-\overline{z}_jJ_{[0,j]})
(\xi_{[0,j]}(\overline{z}_j)+m_{[0,j]}(\overline{z}_j)\pi_{[0,j]}(\overline{z}_j))
\end{split}
\end{equation}
Moreover, if $J^{-\frac{1}{2}}HJ^{-\frac{1}{2}}$ is self-adjoint in $\ell^2$ then the systems 
$\left\{(J^{-\frac{1}{2}}HJ^{-\frac{1}{2}}-\overline{z}_j)^{-1}J^{-\frac{1}{2}}e_0\right\}_{j=0}^{\infty}$ and 
$\{J^{\frac{1}{2}}e_j\}_{j=0}^{\infty}$
are equivalent, that is,
\[
span\{(J^{-\frac{1}{2}}HJ^{-\frac{1}{2}}-\overline{z}_0)^{-1},\dots,(J^{-\frac{1}{2}}HJ^{-\frac{1}{2}}-\overline{z}_k)^{-1}e_0\}=
span\{J^{\frac{1}{2}}e_0,\dots,J^{\frac{1}{2}}e_k\}
\]
for every $k\in\dZ_+$.
\end{lemma}
\begin{proof}
Notice that $\overline{\mathfrak{b}}_j-\overline{z}_j{\mathfrak{d}}_j=0$. Then 
it follows from~\eqref{forCD1}, and~\eqref{forCD2} that
\begin{equation}\label{h2}
\begin{split}
(H-\overline{z}_j J)\pi_{[0,j]}(\overline{z}_j)&=
(H_{[0,j]}-\overline{z}_j J_{[0,j]})\pi_{[0,j]}(\overline{z}_j)
=-(\mathfrak{b}_j-\overline{z}_j\mathfrak{d}_j)\widehat{P}_{j+1}(\overline{z}_j)e_j,
\\
(H-\overline{z}_j J)\xi_{[0,j]}(\overline{z}_j)&=
(H_{[0,j]}-\overline{z}_j J_{[0,j]})\xi_{[0,j]}(\overline{z}_j)
=-(\mathfrak{b}_j-\overline{z}_j\mathfrak{d}_j)\widehat{Q}_{j+1}(\overline{z}_j)e_j+e_0.
\end{split}
\end{equation}
Now, \eqref{h1} is immediate from~\eqref{h2} by taking into account 
\[
m_{[0,j]}(\overline{z}_j)=-\frac{Q_{j+1}(\overline{z}_j)}{P_{j+1}(\overline{z}_j)}
=-\frac{\widehat{Q}_{j+1}(\overline{z}_j)}{\widehat{P}_{j+1}(\overline{z}_j)}.
\]
If $J^{-\frac{1}{2}}HJ^{-\frac{1}{2}}$ is a self-adjoint operator in $\ell^2$ then~\eqref{h1} implies that 
\begin{equation}\label{h1_explicit}
(J^{-\frac{1}{2}}HJ^{-\frac{1}{2}}-\overline{z}_j)^{-1}J^{-\frac{1}{2}}e_0=
J^{\frac{1}{2}}(\xi_{[0,j]}(\overline{z}_j)+m_{[0,j]}(\overline{z}_j)\pi_{[0,j]}(\overline{z}_j)).
\end{equation}
Now, the equivalence follows from~\eqref{h1_explicit} for $j=0,\dots,k$
and the fact that 
$\widehat{Q}_j(\overline{z}_j)+m_{[0,j]}(\overline{z}_j)\widehat{P}_j(\overline{z}_j)\ne 0$ 
for $j=0,\dots, k$. The latter fact immediately follows from~\eqref{wTransform}, \eqref{mformulas},
and the Liouville-Ostrogradsky formula~\eqref{Ostrogr2}.
\end{proof}

\begin{proposition}\label{self}
If the operator $J^{-\frac{1}{2}}HJ^{-\frac{1}{2}}$ is self-adjoint in $\ell^2$ then the corresponding
Nevanlinna-Pick problem~\eqref{NP_ex} has the unique solution
\[
\f(\lambda)=m(\lambda):=((J^{-\frac{1}{2}}HJ^{-\frac{1}{2}}-\lambda I)^{-1}J^{-\frac{1}{2}}e_0,J^{-\frac{1}{2}}e_0).
\]
\end{proposition}
\begin{proof}
Clearly, for every $\lambda\in\dC_+\cup\dC_-$
there exists a sequence
 $r_n(\lambda)\in\span\{J^{\frac{1}{2}}e_0,\dots,J^{\frac{1}{2}}e_n\}\subset\dom(J^{\frac{1}{2}}HJ^{\frac{1}{2}})$ 
such that
\begin{equation}\label{help_for_conv}
\Vert(J^{-\frac{1}{2}}HJ^{-\frac{1}{2}}-\lambda)r_n(\lambda)-J^{-\frac{1}{2}}e_0\Vert\to 0,\quad n\to\infty.
\end{equation}
It follows from Lemma~\ref{lemma1} that
\begin{equation}\label{def_r_n}
r_n(\lambda)=\sum_{k=0}^{n}{c_k(\lambda)}(J^{-\frac{1}{2}}HJ^{-\frac{1}{2}}-\overline{z}_k)^{-1}J^{\frac{1}{2}}e_0.
\end{equation}
Further, let $HJ^{-1}=\int_{\dR}tdE_t$ be a spectral decomposition of $J^{-\frac{1}{2}}HJ^{-\frac{1}{2}}$. Then the
function 
\[
m(\lambda)=\int_{\dR}\frac{d(E_te_0,e_0)}{t-\lambda}=
((J^{-\frac{1}{2}}HJ^{-\frac{1}{2}}-\lambda )^{-1}J^{-\frac{1}{2}}e_0,J^{-\frac{1}{2}}e_0).
\]
is a solution of the Nevanlinna-Pick problem~\eqref{NP_ex}. Really, according to~\eqref{h1}
we have
\[
m(\overline{z}_j)=((J^{-\frac{1}{2}}HJ^{-\frac{1}{2}}-\overline{z}_j )^{-1}J^{-\frac{1}{2}}e_0,J^{-\frac{1}{2}}e_0)_{\ell^2}
=m_{[0,j]}(\overline{z}_j).
\]
Further, due to~\eqref{IP} and~\eqref{mformulas} one easily gets that 
$m(z_j)=w_j$ for $j\in\dZ_+$. Suppose that there is another solution 
$\f_{\rho}(\lambda)=\int_{\dR}\frac{d\rho(t)}{t-\lambda}$. Then we have
\begin{multline*}
\int_{\dR}\left|(t-\lambda)\sum_{k=0}^{n}\frac{c_k(\lambda)}{t-\overline{z}_j}-1\right|^2d\rho(t)=\\
=\int_{\dR}\left|(t-\lambda)\sum_{k=0}^{n}\frac{c_k(\lambda)}{t-\overline{z}_j}-1\right|^2
d(E_te_0,e_0)=\\
=\Vert (J^{-\frac{1}{2}}HJ^{-\frac{1}{2}}-\lambda)\sum_{k=0}^n c_k(\lambda)
(J^{-\frac{1}{2}}HJ^{-\frac{1}{2}}-\overline{z}_j)^{-1}J^{-\frac{1}{2}}e_0-J^{-\frac{1}{2}}e_0\Vert
\to 0,
\end{multline*}
as $n\to\infty$. Now, $1/(t-\lambda)$ is bounded for $t\in\dR$ since
$\lambda\in\dC_+\cup\dC_-$. Thus
\[
\int_{\dR}\left|\sum_{k=0}^{n}\frac{c_k(\lambda)}{t-\overline{z}_k}-\frac{1}{t-\lambda}\right|^2
d\rho(t)\to 0, \quad n\to\infty.
\]
Finally, it follows that
\[
\f_{\rho}(\lambda)=
\lim_{n\to\infty}\int_{\dR}\sum_{k=0}^{n}\frac{c_k(\lambda)}{t-\overline{z}_k}d\rho(t)
\]
is independent of $\rho$. Since $\f_{\rho}$ determines $\rho$ (see for instance~\cite[Chapter~III]{A}), 
all $\rho$'s must be the same.
\end{proof}

\begin{proposition}\label{symmetric}
If the operator $J^{-\frac{1}{2}}HJ^{-\frac{1}{2}}$ is not self-adjoint in $\ell^2$ then the corresponding
Nevanlinna-Pick problem~\eqref{NP_ex} has an infinite number of solutions.
\end{proposition}
\begin{proof}
Since the deficiency indices of $J^{-\frac{1}{2}}HJ^{-\frac{1}{2}}$
are equal it has self-adjoint extensions in $\ell^2$. Let $H_1$ and $H_2$ be two different self-adjoint extensions of 
$J^{-\frac{1}{2}}HJ^{-\frac{1}{2}}$ in $\ell^2$. Then the following two functions
\[ 
\f_1(\lambda)=((H_1-\lambda)^{-1}J^{-\frac{1}{2}}e_0,J^{-\frac{1}{2}}e_0), \quad
\f_2(\lambda)=((H_2-\lambda)^{-1}J^{-\frac{1}{2}}e_0,J^{-\frac{1}{2}}e_0)
\]
are solutions of~\eqref{NP_ex}. Really, according to Lemma~\ref{lemma1} we have
\[
\f_k(\overline{z}_j)=((H_k-\overline{z}_j)^{-1}J^{-\frac{1}{2}}e_0,J^{-\frac{1}{2}}e_0)=
((H_{[0,j]}-\overline{z}_j J_{[0,j]})^{-1}e_0,e_0)=\overline{w}_j
\]
for every $j\in\dZ_+$ and $k=1,2$. Since $\f_k\in\bR_0$, one also has
$\f_k(z_j)=w_j$.

Further, let $\lambda\in\dC_+\setminus\{z_j\}_{j=0}^{\infty}$.
Note, that $g_0=J^{-\frac{1}{2}}e_0\not\in\ran(J^{-\frac{1}{2}}HJ^{-\frac{1}{2}}-\lambda)$. To see this, suppose the contrary
that there exists $x\in\dom(J^{-\frac{1}{2}}HJ^{-\frac{1}{2}}-\lambda)$ such that 
$g_0=(J^{-\frac{1}{2}}HJ^{-\frac{1}{2}}-\lambda)x$
and that $((J^{-\frac{1}{2}}HJ^{-\frac{1}{2}})^*-\overline{\lambda})y=0$. Then 
\[
(g_0,y)=((J^{-\frac{1}{2}}HJ^{-\frac{1}{2}}-\lambda)x,y)
=(x,((J^{-\frac{1}{2}}HJ^{-\frac{1}{2}})^*-\overline{\lambda})y)=0.
\]
We thus see that $(g_0,y)=0$ and $((J^{-\frac{1}{2}}HJ^{-\frac{1}{2}})^*-\overline{\lambda})y=0$.
As a consequence, the coefficients $\widehat{u}_k=(g_k,y)$ of the vector $y\sim\sum_{k=0}^{\infty}\widehat{u}_kf_k$ solve~\eqref{r_rl}
with the initial conditions $\widehat{u}_{-1}=\widehat{u}_0=0$. Therefore, $y=0$, that is, 
$J^{-\frac{1}{2}}HJ^{-\frac{1}{2}}$ is self-adjoint in $\ell^2$.
By hypothesis, this is false, so $J^{-\frac{1}{2}}e_0\not\in\ran(J^{-\frac{1}{2}}HJ^{-\frac{1}{2}}-\lambda)$.
Thus $(H_1-\lambda)^{-1}J^{-\frac{1}{2}}e_0$ and $(H_2-\lambda)^{-1}J^{-\frac{1}{2}}e_0$ are in
$\dom((J^{-\frac{1}{2}}HJ^{-\frac{1}{2}})^*)\setminus\dom(J^{-\frac{1}{2}}HJ^{-\frac{1}{2}})$.
So, we have $(H_1-\lambda)^{-1}J^{-\frac{1}{2}}e_0\ne(H_2-\lambda)^{-1}J^{-\frac{1}{2}}e_0$ because otherwise, according to
the fact that $J^{-\frac{1}{2}}HJ^{-\frac{1}{2}}$ has deficiency indices (1,1) and the von Neumann formulas we would have
$H_1=H_2$.

Let $\eta=(H_1-\lambda)^{-1}J^{-\frac{1}{2}}e_0-(H_2-\lambda)^{-1}J^{-\frac{1}{2}}e_0$. Then one has
$((J^{-\frac{1}{2}}HJ^{-\frac{1}{2}})^*-\lambda)\eta=0$ and, so, 
the coefficients $\widehat{\eta}_k=(g_k,\eta)$ of the vector $\eta\sim\sum_{k=0}^{\infty}\widehat{\eta}_kf_k$
give a solution of~\eqref{r_rl} with the initial conditions 
\[
\widehat{\eta}_{-1}=0,\quad \widehat{\eta}_0=(g_0,\eta).
\]
Since $\eta\ne 0$ we get $(g_0,\eta)\ne 0$.
As a consequence, we have $\f_1\not\equiv\f_2$. To complete the proof it remains to observe
that the function
\[
\f_{\alpha}(\lambda)=\alpha\f_1(\lambda)+(1-\alpha)\f_2(\lambda)
\]
is also a solution of~\eqref{NP_ex} for every $\alpha\in(0,1)$.
\end{proof}
\begin{remark}
It follows from the proof that every self-adjoint extension of the symmetric operator $J^{-\frac{1}{2}}HJ^{-\frac{1}{2}}$ 
generates a solution of the corresponding Nevanlinna-Pick problem. Moreover, by using the standard technique
of theory of extensions of symmetric operators (see~\cite{A},~\cite{DM},~\cite{Mal}), one can get 
the description of all solutions of the Nevanlinna-Pick problem and it will be done elsewhere.
The description of all solutions can be found, for instance, in~\cite{Garnett}.
\end{remark}

The following theorem immediately follows from Propositions~\ref{self} and~\ref{symmetric}.
\begin{theorem}\label{NP_determinate}
The Nevanlinna-Pick problem~\eqref{NP_ex} has a unique solution iff the corresponding operator
$J^{-\frac{1}{2}}HJ^{-\frac{1}{2}}$ is self-adjoint in $\ell^2$.
\end{theorem}
\begin{remark}
Other criteria for the Nevanlinna-Pick problems 
to be determinate can be found in~\cite{Garnett},~\cite{KN}.
It is worth noting that, in the matrix case, the Stieltjes type criteria for 
Nevanlinna-Pick problems to be completely indeterminate
were obtained by Yu.~M.~Dyu\-karev in his second doctorate thesis 
(see~\cite{Dyu05},~\cite{Dyu08}).
\end{remark}

\section{Convergence of multipoint Pad\'e approximants}

In this section we prove a Markov type result on convergence of
multipoint diagonal Pad\'e approximants for $\bR_0$-functions.

At first, let us recall that for the symmetric matrix $J_{[0,j]}^{-\frac{1}{2}}H_{[0,j]}J_{[0,j]}^{-\frac{1}{2}}$
the following estimate holds true
\begin{equation}\label{un_b_for_r}
\Vert(J_{[0,j]}^{-\frac{1}{2}}H_{[0,j]}J_{[0,j]}^{-\frac{1}{2}}-\lambda )^{-1}\Vert\le\frac{1}{|\Im\lambda|},\quad j\in\dZ_+.
\end{equation}

Before showing the convergence result, it is natural to obtain 
the precompactness.

\begin{proposition}\label{prec_R1}
The family 
$\{m_{[0,j]}\}_{j=0}^{\infty}$ is precompact in the topology of locally uniform convergence in $\dC\setminus\dR$.
\end{proposition}
\begin{proof}Let us rewrite the function $m_{[0,j]}$ as follows
\[
m_{[0,j]}(\lambda)=((J_{[0,j]}^{-\frac{1}{2}}H_{[0,j]}J_{[0,j]}^{-\frac{1}{2}}-\lambda )^{-1}J_{[0,j]}^{-\frac{1}{2}}e_0,J_{[0,j]}^{-\frac{1}{2}}e_0).
\]
It follows from the Cauchy-Swarz inequality and~\eqref{dom_inq_tr} that
\begin{equation}\label{pre}
|m_{[0,j]}(\lambda)|=\frac{(J^{-1}_{[0,j]}e_0,e_0)}{|\Im\lambda|}\le \frac{1}{|\Im\lambda|},
\end{equation}t
which, in view of the Montel theorem, implies the precompactness of 
$\{m_{[0,j]}\}_{j=0}^{\infty}$.
\end{proof}

Now we are ready to prove the main result of this section.
\begin{theorem}\label{MarkovMP}
Let a sequence of distinct numbers $\{z_j\}_{j=0}^{\infty}\subset\dC_+$ be given and let
$\f$ be a unique solution of the Nevanlinna-Pick problem~\eqref{NP_ex}. Then
all the multipoint diagonal Pad\'e approximants for $\f$ at
$\{z_0,\overline{z}_0,\dots,z_j,\overline{z}_j,\dots\}$
exist and converge to $\f$ locally uniformly in $\dC\setminus\dR$.
\end{theorem}
\begin{proof}
Proposition~\ref{mformulas} says that the rational function $m_{[0,j]}$ is
the (j+1)th multipoint diagonal Pad\'e approximant. 
Further, according to Theorem~\ref{NP_determinate}, one obviously has that
$J^{-\frac{1}{2}}HJ^{-\frac{1}{2}}$ is self-adjoint in $\ell^2$ and, therefore,
$(J^{-\frac{1}{2}}HJ^{-\frac{1}{2}}-\lambda)^{-1} $ is bounded for $\lambda\in\dC\setminus\dR$.
Let $\psi$ be a finite sequence, that is, $\psi=(\psi_1,\dots,\psi_k,0,0,\dots)^{\top}$.
Then
\[
(H-\lambda J)\psi=(H_{[0,j]}-\lambda J_{[0,j]})\psi=\phi
\]
for sufficiently large $j\in\dZ_+$ and $\phi$ is also a finite sequence. Further,
one obviously has
\begin{equation}\label{srcon}
\left((J^{-\frac{1}{2}}HJ^{-\frac{1}{2}}-\lambda )^{-1}J^{-\frac{1}{2}}\phi, J^{-\frac{1}{2}}e_0\right)=
\lim_{j\to\infty}\left((J_{[0,j]}^{-\frac{1}{2}}H_{[0,j]}J_{[0,j]}^{-\frac{1}{2}}-\lambda )^{-1}J_{[0,j]}^{-\frac{1}{2}}\phi,J_{[0,j]}^{-\frac{1}{2}}e_0\right).
\end{equation}
In particular, formula~\eqref{srcon} is valid for
\[
\phi_n=(HJ^{-\frac{1}{2}}-\lambda J^{\frac{1}{2}})r_n(\lambda),
\]
where $r_n$ is defined by~\eqref{def_r_n}. So, due to~\eqref{help_for_conv}
we have that 
\begin{equation}\label{phi_n_conv1}
J^{-\frac{1}{2}}\phi_n\to J^{-\frac{1}{2}}e_0\quad\text{as}\quad n\to\infty.
\end{equation} 
Moreover, the vectors $\phi_n$ satisfy the following relation
\begin{equation}\label{phi_n_conv2}
J^{-\frac{1}{2}}_{[0,j]}\phi_n\to J^{-\frac{1}{2}}_{[0,j]}e_0\quad\text{as}\quad n\to\infty
\end{equation} 
for $j\in\dZ_+$. To see the latter relation, note that~\eqref{phi_n_conv1} implies
\[
(J^{-\frac{1}{2}}\phi_n,\eta)\to (J^{-\frac{1}{2}}e_0,\eta)\quad\text{as}\quad n\to\infty
\]
for every $\eta\in\ell^2$. Putting $\eta=J^{\frac{1}{2}}J^{-\frac{1}{2}}_{[0,j]}e_k$, $k=0,\dots,j$,
we get~\eqref{phi_n_conv2} from the fact that, in finite-dimensional spaces, the weak convergence is
equivalent to the strong one.
Now, taking into account~\eqref{un_b_for_r}, \eqref{srcon},~\eqref{phi_n_conv1}, and~\eqref{phi_n_conv2}, 
we obtain that~\eqref{srcon} holds true for $\phi=e_0$, that is,
\[
m_{[0,j]}(\lambda)\to m(\lambda)=\f(\lambda)=
((J^{-\frac{1}{2}}HJ^{-\frac{1}{2}}-\lambda )^{-1}J^{-\frac{1}{2}}e_0,J^{-\frac{1}{2}}e_0)
\] 
for any $\lambda\in\dC\setminus\dR$. Finally, the statement of the theorem follows from
the precompactness and the Vitali theorem.
\end{proof}
\begin{remark}
In the case when $\f\in\bR[\alpha,\beta]$ and 
the interpolation points stay away from $[\alpha,\beta]$, 
an analog of the Markov theorem for multipoint diagonal Pad\'e approximants
is well known~\cite{GL},~\cite{StTo} 
(see also~\cite{DZ09} where the operator approach was presented).
In the case when the interpolation points belong to $[-\infty,0)$,
the locally uniform convergence of multipoint Pad\'e approximants
for $\f\in\bR[0,+\infty)$ was proved under the Carleman 
type condition~\cite{Lo78} (see also~\cite{Lo85} where results 
in this direction are reviewed).
It should be also remarked that there are some results on convergence of multipoint 
Pad\'e approximants for rational perturbations of the Cauchy transforms of some complex
measures~\cite{BYa09},~\cite{BYa10}.
\end{remark}
It is a standard fact that the following condition
\begin{equation}\label{BC}
\sum_{k=0}^{\infty}\frac{\Im z_k}{|z_k+i|^2}=+\infty
\end{equation}
implies the determinacy of the corresponding Nevanlinna-Pick problem in $\bR_0$~\cite{Garnett},~\cite{KN}.
Thus, the underlying operator $J^{-\frac{1}{2}}HJ^{-\frac{1}{2}}$ is self-adjoint in $\ell^2$.

\begin{corollary}
If the given sequence $\{z_j\}_{j=0}^{\infty}$ satisfies~\eqref{BC}
then for every $\f\in\bR_0$ all the multipoint diagonal Pad\'e approximants for $\f$ at
$\{z_0,\overline{z}_0,\dots,z_j,\overline{z}_j,\dots\}$
exist and converge to $\f$ locally uniformly in $\dC\setminus\dR$.
\end{corollary}

\begin{remark}
First, note that~\eqref{BC} is sufficient for the Nevanlinna-Pick problem in $\bR_0$ to be determinate
but not necessary (see~\cite[Chapter~IV, Example~4.2]{Garnett}). It should be also noted 
that, under the Szeg\"o condition and the negation of the Blashcke type condition,
the locally uniform convergence of multipoint diagonal Pad\'e approximants for
$\f\in\bR[\alpha,\beta]$ was proved in ~\cite{St00} (see also~\cite{BaKu}).
\end{remark}

Now, we are also able to adapt Theorem~\ref{WeylSolution} for the self-adjoint case.

\begin{theorem} If $J^{-\frac{1}{2}}HJ^{-\frac{1}{2}}$ is self-adjoint in $\ell^2$ then
for every $\lambda\in\dC_+\setminus\{z_k\}_{k=0}^{\infty}$ there holds
\[
\sum_{k=0}^{\infty}|m(\lambda)(\widehat{P}_k(\lambda)+\mathfrak{d}_{k-1}\widehat{P}_{k-1}(\lambda))+
\widehat{Q}_k(\lambda)+\mathfrak{d}_{k-1}\widehat{Q}_{k-1}(\lambda)|^2=
\frac{m(\lambda)-\overline{m(\lambda)}}{\lambda-\overline{\lambda}}.
\]
 \end{theorem}
\begin{proof}
According to Theorem~\ref{MarkovMP} and~\eqref{Ach7}, we have that ${\bf K}_j(\lambda)\to m(\lambda)$ as $j\to\infty$. 
Now, the statement directly follows from  Corollary~\ref{Kembed} and the inequality~\eqref{WeylBall}.
\end{proof}

\noindent{\bf Acknowledgments}. 
This work was partially done when I was visiting the University of Sciences and Technologies of Lille 
as a postdoc. I would like to express my gratitude to Professor B.~Beckermann for
organizing my visit, the hospitality, and many useful discussions, which stimulated me to write this paper.
I am grateful to Professor A.S.~Zhedanov for numerous discussions, which improved the paper. 
Besides, I am deeply indebted to A.~Kostenko and K.~Simonov for the careful reading of the manuscript
and giving considerable comments.
Finally, I would also like to thank the anonymous referees for helpful remarks and suggestions. 

\end{document}